\documentclass[preprint, 3p]{elsarticle}

\usepackage[english]{babel}
\usepackage[utf8]{inputenc}
\usepackage{amsfonts} 
\usepackage{amsmath,hhline,epsfig,latexsym}
\usepackage{amsthm,amscd}
\usepackage{amssymb}
\usepackage{graphicx,color}
\usepackage{multicol}
\usepackage{amsmath}
\usepackage{verbatim}
\usepackage{bbold}
\usepackage{dsfont}
\usepackage{mathrsfs}
\usepackage{caption}
\usepackage[bf]{subfigure}
\usepackage{mathrsfs}
\usepackage{float}
\usepackage{hyperref}
\usepackage[ruled,vlined]{algorithm2e}

\usepackage{color}

\renewcommand{\thefootnote}

\def\_#1{{\lower 0.7ex\hbox{}}_{#1}}

\def\Om{\Omega}%
\def\om{\omega}%

\newtheorem{lemma}{Lemma}[section]
\newtheorem{proposition}{Proposition}[section]
\newtheorem{theorem}{Theorem}[section]
\newtheorem{corollary}{Corollary}[section]

\newtheorem{definition}{Definition}[section]

\newcommand{\beq}{\begin{eqnarray}}
\newcommand{\eeq}{\end{eqnarray}}
\newcommand{\beqa}{\begin{eqnarray*}}
\newcommand{\eeqa}{\end{eqnarray*}}

\begin{document}

\begin{frontmatter}

\title{On Pareto Equilibria for Some Bi-Objective Optimal Control Problems}

\author[UESPI]{P. P. de Carvalho }
\ead{pitagorascarvalho@gmail.com}

\author[USevilla]{E. Fern\'andez-Cara}
\ead{cara@us.es}

\author[UFF]{J. L\'imaco}
\ead{jlimaco@vm.uff.br}

\author[UFERSA]{D. Menezes}
\ead{denilson.menezes@ufersa.edu.br}

\author[UFF]{Y. Thamsten}
\ead{ythamsten@id.uff.br}

\address[UESPI]{Universidade Estadual do Piau\'i, Teresina-PI, Brasil}
\address[USevilla]{Dpto. E.D.A.N., Universidad de Sevilla, Aptdo. 1160, 41080 Sevilla, Spain}
\address[UFF]{Instituto de Matem\'atica e Estat\'istica, Universidade Federal Fluminense, RJ, Brasil}
\address[UFERSA]{Departamento de Ci\^encia e Tecnologia, Universidade Federal Rural do Semi-\'Arido, Cara\'ubas-RN, Brasil}

\date{}


\begin{abstract}
  We investigate Pareto equilibria for bi-objective optimal control problems. Our framework comprises the situation in which an agent acts with a distributed control in a portion of a given domain, and aims to achieve two distinct (possibly conflicting) targets. We analyze systems governed by linear and semilinear heat equations and also systems with multiplicative controls. We develop numerical methods relying on a combination of finite elements and finite differences. We illustrate the computational methods we develop via numerous experiments. 
\end{abstract}


\end{frontmatter}



\section{Introduction} \label{sec:Intro}

Popular models describing a wide range of natural phenomena comprise a system governed by a set of differential equations. Complex problems usually have infinitely many degrees of freedom, leading us to study situations where the dynamics of the state variables is described by Partial Differential Equations (PDE's). A rich class of problems arises when we are not only interested in the observed unaffected system evolution (as a result, e.g., of physical laws of nature), but we want to investigate how agents act upon it to attain a desired behavior.

In a competitive multi-agent setting, i.e., when many competing rational agents are interacting while trying to influence the dynamics of a system, the proper notion of equilibrium is that of Nash. In practice, we encounter such scenarios, e.g., on environmental problems, see \cite{diaz2002neumann} and the references therein. A precise example is the one described in \cite{diaz2004approximate}, where there is a resort lake containing chemical substances governed by diffusive equations. Many plants are located in portions of it, and decide their corresponding targets. Also allowing for the presence of a manager (i.e., an agent having the first-mover advantage), the paper \cite{diaz2004approximate} treats Stackelberg-Nash equilibria for this problem. Since then, this research line flourished considerably, see \cite{araruna2018stackelberg,araruna2017new,araruna2015stackelberg,carreno2019stackelberg,de2020numerical,guillen2013approximate,limaco2009remarks}, to name a few other works in this direction. Some examples of advances on the numerical computation of Nash equilibria are \cite{borzi2013formulation,de2020some,dreves2016nash,dreves2016jointly,rahman2015fem,ramos2002pointwise,ramos2019nash,ramos2007nash}.

The multi-objective single-player problem comprises the situation in which an agent looks at several different targets, possibly conflictive, and an optimal strategy of action is searched for in a suitable overall sense. The mathematical economist V. Pareto proposed the following notion of optimality in this setting, see \cite{pareto1964cours}: a control action is optimal if any change in it cannot lead to an improved performance of all the objectives. We expect that a rational agent envisaging to optimize such a set of objectives would have no incentive to deviate from such a strategy. The investigation of Pareto optimal controls is quite important, e.g., from the viewpoint of the development of public policies. For a survey of recent trends on multi-objective optimal control problems, see \cite{peitz2018survey}. In the paper \cite{logist2010fast}, the authors analyze the computation of Pareto fronts in models governed by ODEs and subject to various types of constraints. Some works considering Pareto optimality in the context of diffusive systems are \cite{baradaran2009optimal,banirazi2020heat,chen2014fast,damavandi2017modeling}.

The study of Pareto optimal controls for distributed PDE systems goes back to \cite{lions1987pareto}. Some other advances afterwards are \cite{bahaa2003quadratic,nakoulima2003pareto}. In \cite{alvarez2010multi}, the authors consider an application of Pareto optimality to a problem of wastewater management. The controls they consider are pointwise and the concentration of the pollutant in question is governed by a diffusive PDE having an appropriate advection term. A recent development, in \cite{desilles2019pareto}, is a new approach to characterize the Pareto front in a Hamilton-Jacobi framework.

For practical problems, being able to numerically compute the Pareto optimal strategies is an obvious demand. In \cite{ramos2002nash}, the authors use conjugate gradient algorithms to compute Nash equilibria when the system is modeled by linear parabolic equations. P. Carvalho and E. Fern\'andez-Cara developed fixed point methods to compute Nash and Pareto equilibria for bi-objective problems under linear and semilinear heat equations in \cite{carvalho2019computation} and, in \cite{de2020computation}, together with J. L\'imaco, these problems were considered for models governed by wave equations. The approach in the work \cite{alvarez2010multi} is to propose an algorithm based in a characteristics-finite element discretization. The authors of \cite{fernandez2020theoretical} developed algorithms for the computation of Pareto optimal strategies for stationary Navier-Stokes models of equations. 

In the present paper, we concentrate in three classes. Namely, assuming a scalar state variable $u=u(x,t),$ together with a scalar control $v=v(x,t),$ where $x$ and $t$ are respectively the spatial and time variable, we consider linear heat equations,
$$
u_t - \Delta u = v,
$$
semilinear heat equations
$$
u_t - \Delta u + F(u) = v,
$$
as well as bilinearly controlled heat equations
$$
u_t - \Delta u = vu.
$$
As in \cite{carvalho2019computation}, we consider bi-objective problems. Thus, we introduce two square-integrable targets $u_1^T = u_1^T(x)$ and $u_2^T = u_2^T(x)$ corresponding optimization criteria $J_i : \mathcal{U} \rightarrow \mathbb{R},$ $i\in \left\{1,2\right\},$ with
\begin{equation} \label{def:ObjFuncs}
    J_i\left( v \right) := \int_{\mathcal{O}_i} \left| u(T) - u_i^T \right|^2\,dx + \frac{\mu}{2} \iint_{\om\times (0,T)} \left| v\left(x,t\right) \right|^2\,dx\,dt.
\end{equation}
Here, $\mathcal{U}$ is a suitable admissible control set, the $\mathcal{O}_i$ are the regions where the agent desires to drive $u(T,\cdot)$ close to $u_i^T$ and $\om$ is fixed a region containing the supports of the strategies under consideration. The set $\mathcal{U}$ is assumed to be closed in $L^2\left(\om \times \left(0,T\right) \right)$ and convex. As we previously discussed, we notice here that, as long as the overlapping region $\mathcal{O}_1 \cap \mathcal{O}_2$ is non-empty, the two objectives may fall in conflict inside.

Let us recall the notion of Pareto optimality that we are going to employ.
\begin{definition} \label{def:ParetoEq}
A control $\hat{v} \in \mathcal{U}$ is called a Pareto equilibrium if there is no other $v \in \mathcal{U}$ such that
$$
J_1\left(v\right) \leqslant J_1\left( \hat{v} \right) \text{ and } J_2\left(v\right) \leqslant J_2\left( \hat{v} \right),
$$
with at least one of these inequalities holding strictly.
\end{definition}
A related notion, which proves to be technically useful to characterize Pareto optimal strategies, is that of Pareto quasi-equlibria. Let us be more precise:
\begin{definition} \label{def:ParetoQuasiEq}
A control $\hat{v}$ is called a Pareto quasi-equilibrium if there exists $\alpha \in \left[0,1\right]$ such that the G\^ateaux derivatives of $J_1$ and $J_2$ at $\hat{v}$ in the direction of $w-\hat{v}$ satisfy
$$
\left\langle \alpha J_1^\prime\left( \hat{v} \right) + \left(1-\alpha \right) J_2^\prime\left(\hat{v} \right), w-\hat{v} \right\rangle \geqslant 0,
$$
for every $w \in \mathcal{U}.$
\end{definition}

Let us describe the outline of the paper:
\begin{itemize}
    \item In Section \ref{sec:Linear}, we deal with models governed by linear heat equations. We give a complete picture of the Pareto front for $\mu\geqslant 0.$
    \item In Section \ref{sec:Semilinear}, we consider systems driven by semilinear parabolic equations. We deal with sufficiently smooth semilinearities, with suitable boundedness assumptions. The situation for $\mu = 0$ is similar to that of the linear case, whereas for $\mu > 0$ we completely characterize the Pareto front assuming that the spatial dimension is $\leqslant 3$ and $\mu$ is sufficiently large.
    \item In Section \ref{sec:bilinear}, we consider multiplicative controls. If the class of controls comprise square-integrable controls, with no further general restrictions, we are once again apt to characterize the Pareto front for $\mu$ large enough, but now assuming the spatial dimension to be at most two. If we work with uniformly essentially bounded controls then, up to dimension three, we can also provide suitable descriptions of the Pareto optimal strategies.
    \item In Section \ref{sec:Numerics1}, we present the algorithms we will use to obtain the solutions computationally. We prove the convergences of all of them. For the linear problem, we employ a conjugate gradient algorithm. For the semilinear model, we propose a fixed point iterative method for $\mu$ large enough and a Newton-Raphson method for general $\mu,$ assuming the initial data to be close enough to the targets. In the framework of multiplicative controls, we employ a gradient descent algorithm if $\mu$ is large and the spatial dimension does not exceed two and a fixed point iterative algorithm for any positive $\mu$ and spatial dimension not greater than three. 
    \item In Section \ref{sec:Numerics2}, we provide numerical illustrations of the previous algorithms.
    \item In Section \ref{sec:Conclusions}, we present our conclusions and give some additional comments.
\end{itemize}

\section{The linear case} \label{sec:Linear}

Let us consider an integer $N \geqslant 1,$ and an open bounded set $\Omega \subset \mathds{R}^N$ with smooth boundary $\Gamma = \partial\Omega.$ We write $Q := \Om \times \left(0,T\right)$ and $\Sigma := \partial\Om \times \left(0,T\right).$ We assume that $\mathcal{O}_1$ and $\mathcal{O}_2$ are two open subsets of $\Omega$ satisfying $\mathcal{O}_1 \cap \mathcal{O}_2 \neq \emptyset$ and we act through a distributed control spatially supported by an open subset $\om$ of $\Om.$  To discard trivial cases, let us assume henceforth that
$$
\mathcal{O}_1\neq \mathcal{O}_2 \text{ and } \mathcal{O}_1 \Delta \mathcal{O}_2 := \left( \mathcal{O}_1\backslash \mathcal{O}_2 \right) \cup \left( \mathcal{O}_2\backslash \mathcal{O}_1 \right) \subset \Omega\setminus \om.
$$

As usual, the symbol $C$ will stand for a generic positive constant. The norm and scalar product in $L^2(\Omega)$ will be respectively denoted by $\|\cdot\|$ and $(\cdot,\cdot)$.

We will assume that $v \in L^2(\omega\times(0,T))$ and $u_0 \in L^2(\Omega).$ In this section, we are concerned with the initial-boundary problem
\begin{equation}\label{pareto-1}
\begin{cases}
        u_t -\Delta{u}= v \mathds{1}_{\omega}, \quad &\mbox{in} \ \Omega\times(0,T), \\
        u = 0, \quad &\mbox{on} \ \partial\Omega\times(0,T),\\
        u(0) = u_0,\quad &\mbox{in} \ \Omega,
\end{cases}
\end{equation}
where we denoted by $\mathds{1}_{\omega}$ the characteristic function of $\omega$ and $u=u(x,t)$ is the unknown state variable.

The following result is well known (see for instance \cite{evans10}):

\begin{lemma}\label{lemma2.1}
Given $v \in L^2\left( \om \times \left(0,T\right) \right)$ and $u_0 \in L^2(\Om),$ there exists exactly one function $u \in C^0([0,T];L^2(\Om))\cap L^2(0,T;H^1_0(\Om)),$ with $u_t \in L^2(0,T;H^{-1}(\Om)),$ which is a weak solution to \eqref{pareto-1}. Furthermore,
$$
\|u\|_{L^\infty(0,T;L^2(\Om))} + \|u\|_{L^2(0,T;H^1_0(\Om))} + \|u_t\|_{L^2(0,T;H^{-1}(\Om))} \leqslant C\left( \|u_0\| + \|v\|_{L^2\left(\om \times\left(0,T\right)\right)} \right).
$$
\end{lemma}

We fix two uncontrolled trajectories $u_1$ and $u_2$ of \eqref{pareto-1}, initiated at distinct states $u_{0,1}$ and $u_{0,2},$ that is,
\begin{equation}\label{pareto-2}
  \begin{cases}
    (u_i)_t -\Delta{u_i}= 0, \quad &\mbox{in} \ \Omega\times(0,T), \\
    u_i = 0, \quad &\mbox{on} \ \partial\Omega\times(0,T),\\
    u_i(0) = u_{0i},\quad &\mbox{in} \ \Omega 
  \end{cases}
\end{equation}
for $i=1,2$ and
$$
u_{0,1} \not\equiv u_{0,2}.
$$

The considered objective functionals are those in \eqref{def:ObjFuncs}, with $u_i^T := u_i(\cdot,T).$ Intuitively, for each $i,$ the functional $J_i$ is constructed as follows: it measures the distance from the controlled state $u$ at the terminal time to the corresponding uncontrolled trajectory final value plus some additional penalization related to the cost of the control. Thanks to the constant $\mu$ we can model the performance-cost trade-off. 

Recall that we assume that the admissible set $\mathcal{U}$ is closed in $L^2\left( \om \times \left(0,T\right)\right)$ and convex. We will denote by $P_\mathcal{U}:\ L^2(\om\times (0,T))\mapsto \mathcal{U}$ the associated orthogonal projector. We remark that $J_1$ and $J_2$ are both convex, and that the convexity is strict as long as $\mu>0.$ Moreover, if $\mu>0$ and $\mathcal{U}$ is unbounded in $L^2\left( \om \times \left(0,T\right) \right),$ we have
$$
J_i(v) \rightarrow +\infty\ \text{as}\ \|v\|_{L^2\left( \om \times \left(0,T\right) \right)} \rightarrow +\infty.
$$

We now turn to the investigation of Pareto optimality. Note that we implicitly assume in Definition \ref{def:ParetoQuasiEq} that $J_1$ and $J_2$ are differentiable, in a suitable sense. Presently, this is the case, as is well-known from standard optimal control theory. Indeed, for every $v,w \in \mathcal{U},$ we have $v+ \epsilon(w-v) \in \mathcal{U},$ as long as $0 < \epsilon \leqslant 1,$ and, furthermore,
\begin{align}\label{pareto-3}
	\begin{split}
		\left\langle J_i^\prime(v),w -v \right\rangle &= \lim_{\epsilon \to 0} \frac{1}{\epsilon}\left[ J_i(v + \epsilon(w-v)) - J_i(\hat{v}) \right] \\
		&=\int_{\mathcal{O}_i} (u(T) - u_i(T))z(T) \,dx + \mu\iint_{\om \times \left(0,T\right) } v (w-v)\,dx\,dt \\
		&= \int_{\omega\times(0,T)} \left(\varphi_i + \mu v \right)(w-v) \ dxdt ,
	\end{split}
\end{align}
where $z$ is the solution of the linearized equation,
\begin{equation}\label{pareto-4}
  \begin{cases}
    z_t -\Delta{z}= (w - v)\mathds{1}_{\omega}, \quad &\mbox{in} \ \Omega\times(0,T), \\
    z = 0, \quad &\mbox{on} \ \partial\Omega\times(0,T),\\
    z(0) = 0,\quad &\mbox{in} \ \Omega,
  \end{cases}
\end{equation}
whereas, for $i \in \left\{ 1,2 \right\},$ the function $\varphi_i$ solves the adjoint equation
\begin{equation}\label{pareto-5}
\begin{cases}
    -\varphi_{i,t} -\Delta{\varphi_i}= 0, \quad &\mbox{in} \ \Omega\times(0,T), \\
    \varphi_i = 0, \quad &\mbox{on} \ \partial\Omega\times(0,T),\\
    \varphi_i(T) = (u(T) - u_i(T))\mathds{1}_{\mathcal{O}_i} ,	\quad &\mbox{in} \ \Omega.
\end{cases}
\end{equation}
This fact promptly implies a characterization of the gradient $J_i^\prime$ of $J_i,$ and also of Pareto quasi-equilibria:
\begin{corollary} \label{cor:CharacGrads_Linear}
    Let us suppose $\mathcal{U} = L^2\left( \om \times \left(0,T\right) \right).$ Then, for each $i \in \left\{1,2\right\},$ the gradient of $J_i$ at $v \in L^2\left( \om \times \left(0,T\right)\right)$ is given by
    $$
    J_i^\prime(v) = \left.{\varphi_i}\right|_{\om \times (0,T)} + \mu v,
    $$ 
    where $\varphi_i$ solves \eqref{pareto-5}. 
\end{corollary}
From the general form \eqref{pareto-3} of the G\^ateaux derivatives, we provide a characterization of Pareto quasi-equilibria:
\begin{corollary} \label{cor:Charact_ParetoQuasiEq_Linear}
    If $\mu>0,$ a control $\hat{v}\in \mathcal{U}$ is a Pareto quasi-equilibrium if, and only if, there exists $\alpha \in \left[0,1\right]$ such that
    \begin{equation}\label{paretoeq_linear_mupos}
        \hat{v} = P_\mathcal{U} \left( - \frac{1}{\mu} \left[ \alpha\varphi_1 + (1-\alpha )\varphi_2 \right]\bigg|_{\om \times (0,T)} \right).
   \end{equation}
    If $\mu=0,$ this happens if, and only if, one has
    \begin{equation}\label{paretoeq_linear_mu0}
        P_\mathcal{U} \left(\left[\alpha\varphi_1 + (1-\alpha )\varphi_2\right]\bigg|_{\om \times (0,T)} \right) = 0.
    \end{equation}
\end{corollary}

We now provide results that relate the notions of Pareto equilibria and quasi-equilibria:
\begin{proposition} \label{prop:ParetoEqToQuasiEq}
If $\hat{v}$ is a Pareto equilibrium, then it is a Pareto quasi-equilibrium.
\end{proposition}
\begin{proof}
If $\hat{v}$ is a Pareto equilibrium, then
$$
J_1\left( \hat{v} \right) \leqslant J_1\left( v \right),
$$
for every $v \in \mathcal{U}$ such that $J_2\left( v \right) \leqslant E_2 := J_2\left( \hat{v} \right).$ From the Kuhn-Tucker Criterion, see \cite[Theorems 9.2-3,4]{ciarlet1989introduction}, we deduce that there exist $\lambda_1,\lambda_2 \in \mathds{R},$ not both zero, such that 
$$
\lambda_1 J_1^\prime \left( \hat{v} \right) + \lambda_2 J_1^\prime\left( \hat{v} \right) = 0.
$$
Moreover, if $\lambda_1 \neq 0,$ then $\lambda_2 / \lambda_1 \geqslant 0.$ Hence, $\hat{v}$ is a Pareto quasi-equilibrium.
\end{proof}

\begin{proposition} \label{prop:ParetoQuasiEqToEq}
If $\hat{v}$ is a Pareto quasi-equilibrium and $\alpha \in \left(0,1\right),$ then it is a Pareto equilibrium. If we assume that $\mu > 0,$ then the result also holds for $\alpha=0$ and $\alpha = 1.$
\end{proposition}
\begin{proof}
Let us assume that $\hat{v} \in L^2\left( \om \times \left(0,T\right) \right)$ is not a Pareto equilibrium. Then, for some $v \in \mathcal{U},$ one has
$$
J_i\left( v \right) \leqslant J_i\left( \hat{v} \right),\, i\in \left\{1,2\right\},
$$
one of the inequalities being strict. Therefore, since $\alpha \in \left(0,1\right),$ one has
$$
\alpha J_1\left( v \right) + \left(1 - \alpha \right) J_2\left( v \right) < \alpha J_1\left( \hat{v} \right) + \left(1 - \alpha \right) J_2\left( \hat{v} \right).
$$
Consequently, $\hat{v}$ cannot be a Pareto quasi-equilibrium corresponding to this value of $\alpha.$ 

Assuming that $\mu > 0,$ we see that a Pareto quasi-equilibrium $\hat{v}$ corresponding to $\alpha = 0$ (respectively, $\alpha = 1$) must be the unique minimizer of $J_2$ (respectively, $J_1$). Thus, $\hat{v}$ is a Pareto equilibrium. 
\end{proof}

We now turn to the analysis of Pareto quasi-equilibria under the assumption that $\mu = 0.$
\begin{theorem} \label{thm:mu0case_linear}
    Let us assume that $\mu = 0$ and $\mathcal{U}$ has the following continuation property:
    $$
    P_{\mathcal{U}}\left( \varphi \right) = 0 \text{ implies } \varphi = 0 \text{ a.e. in } Q,
    $$
    for every solution $\varphi$ of a homogeneous linear heat equation. Then, there cannot exist a Pareto equilibrium $\hat{v}$ such that $J_1\left( \hat{v} \right) > 0$ and $J_2\left( \hat{v} \right) > 0.$ On the other hand, if 
    $$
    \mathcal{U} \supseteq B_R := \left\{ v \in L^2\left(\om \times \left(0,T\right) \right): \|v\|_{L^2\left(\om \times \left(0,T\right) \right)} \leqslant R \right\}
    $$ 
    for some $R>0$ that is large enough, there exist Pareto equilibria $\hat{v}$ with $J_1\left( \hat{v} \right) = 0$ or $J_2\left(\hat{v}\right) = 0.$ 
\end{theorem}
\begin{proof}
From Corollary \ref{cor:Charact_ParetoQuasiEq_Linear}, we have that
$$
\widetilde{\varphi} := \alpha \varphi_1 + \left( 1- \alpha \right) \varphi_2
$$
must satisfy $P_{\mathcal{U}}\left( \widetilde{\varphi}|_{\om \times (0,T)} \right) = 0.$ Since $\widetilde{\varphi}$ solves a homogeneous heat equation, from the continuation property we have assumed, it follows that that $\widetilde{\varphi} \equiv 0.$ Therefore,
\begin{equation} \label{eq:MuZeroCondn}
    \alpha \left( u(T) - u_1(T) \right)\mathds{1}_{\mathcal{O}_1} + \left(1-\alpha\right) \left( u(T) - u_2(T) \right)\mathds{1}_{\mathcal{O}_2} = 0 \text{ a.e. in } \Om.
\end{equation}
If $\alpha = 0,$ then \eqref{eq:MuZeroCondn} reduces to $u(T) = u_2(T)$ a.e. in $\mathcal{O}_2;$ thus, we have $J_2\left( \hat{v} \right) = 0.$ Likewise, $\alpha = 1$ implies $u(T) = u_1(T)$ a.e. in $\mathcal{O}_1,$ whence it follows that $J_1\left( \hat{v} \right) = 0.$

Let us assume $\alpha \in \left( 0,1 \right).$ In this case, \eqref{eq:MuZeroCondn} gives 
$$
u(T) = u_1(T) \text{ a.e. in } \mathcal{O}_1 \setminus \mathcal{O}_2
$$
and
$$
u(T) = u_2(T) \text{ a.e. in } \mathcal{O}_2 \setminus \mathcal{O}_1.
$$
The fact that both $u_1(T)$ and $u_2(T)$ possess analytic versions in $\Omega \backslash \overline{\om}$ allows us to conclude that $u_1(T) = u_2(T),$ which is in contradiction with our assumptions. This proves the first part.

Now, let us assume that $\mathcal{U} \supseteq B_R$ for some $R>0$ such that there exists $\hat{v}_2 \in B_R$ whose associated state $\hat{u}$ satisfies $\hat{u}(T) = u_2(T).$ We then have $J_2\left(\hat{v}_2\right) =0,$ whereas $J_1(\hat{v}_2) = \|u_1(T) - u_2(T)\|^2_{L^2(\mathcal{O}_1)}$. Note that $J_1(v)=J_1(\hat{v})$ for any other control such that $J_2(v)=0$.Hence, $\hat{v}$ is a Pareto equilibrium satisfying $J_2(\hat{v}) = 0.$ An analogous construction provides a Pareto equilibrium $\hat{v}_1 \in L^2\left( \om \times \left(0,T\right) \right)$ such that $J_1(\hat{v}_1) = 0.$ 
\end{proof}

For $\mu > 0,$ the set of Pareto equilibria for our objective functional is richer. We describe it in the following result:
\begin{theorem}
    If $\mu > 0$ and $\mathcal{U} = L^2\left( \om \times \left(0,T\right) \right),$ there exists a family $\left\{ \hat{v}_\alpha \right\}_{\alpha \in \left[0,1\right]}$ of Pareto equilibria.
\end{theorem}
\begin{proof}
By \eqref{paretoeq_linear_mupos}, we deduce that a control $\hat{v}$ is a Pareto equilibrium for $J_1$ and $J_2$ if, and only if, there exists $\alpha \in \left[ 0,1 \right]$ such that
\begin{equation}\label{Pareto-L}
 \hat{v}_\alpha = -\frac{1}{\mu}\left[ \alpha\varphi_1 + (1-\alpha)\varphi_2 \right]\bigg|_{\om \times (0,T)}.
\end{equation}
Let us put $u=\overline{u}+w$ and $\varphi_i = \overline{\varphi_i} + \psi_i$, where
\begin{equation}
    \begin{cases}
      \overline{u}_t -\Delta{\overline{u}}= 0, \quad &\mbox{in} \ \Omega\times(0,T), \\
      \overline{u} = 0, \quad &\mbox{on} \ \partial\Omega\times(0,T),\\
      \overline{u}(0) = u_0,\quad &\mbox{in} \ \Omega,
    \end{cases}
\end{equation}

\begin{equation}
    \begin{cases}
      -\overline{\varphi}_{i,t} -\Delta{\overline{\varphi_i}}= 0, \quad &\mbox{in} \ \Omega\times(0,T), \\
      \overline{\varphi_i} = 0, \quad &\mbox{on} \ \partial\Omega\times(0,T),\\
      \overline{\varphi_i}(T) = (\overline{u}(T) - u_i(T))\mathds{1}_{\mathcal{O}_i} ,	\quad &\mbox{in} \ \Omega .
    \end{cases}
\end{equation}

We also set $\psi := \alpha\psi_1 + (1-\alpha)\psi_2,$ in such a way that
\begin{equation}\label{pareto-11}
	\begin{cases}
	  -\psi_{t} -\Delta{\psi}= 0, \quad &\mbox{in} \ \Omega\times(0,T),  \\
      \psi = 0, \quad &\mbox{on} \ \partial\Omega\times(0,T),\\
      \psi(T) = w(T)(\alpha \mathds{1}_{\mathcal{O}_1} + (1-\alpha)\mathds{1}_{\mathcal{O}_2}) , \quad &\mbox{in} \ \Omega .
	\end{cases}
\end{equation}	
	
With these notation, \eqref{Pareto-L} reads as follows:
\begin{equation}\label{pareto-12}
    \hat{v}_\alpha = -\frac{1}{\mu}\left[ \alpha\overline{\varphi}_1 + (1-\alpha)\overline{\varphi}_2 + \alpha\psi_1 + (1-\alpha)\psi_2 \right]\Big|_{\omega\times(0,T)}.
\end{equation}
Therefore, after introducing the function $f_\alpha := \alpha\overline{\varphi}_1 + (1-\alpha)\overline{\varphi}_2$ and setting
\begin{equation}\label{pareto-13}
	\Lambda_\alpha(v) := \alpha\psi_1 + (1-\alpha)\psi_2,
\end{equation}
we obtain that a Pareto equilibrium $\hat{v} \in L^2\left( \om \times \left(0,T\right) \right)$ in the present context is characterized as a solution of the operator equation
\begin{equation}\label{pareto-14}
    \hat{v} + \frac{1}{\mu}\Lambda_\alpha\left( \hat{v} \right) = -\frac{1}{\mu}f_\alpha.
\end{equation}

It is straightforward to check that $\Lambda_\alpha$ is a linear compact operator on $L^2\left( \om \times \left(0,T\right) \right)$. Furthermore,
\begin{align*}
    \left( \Lambda_\alpha(v),v \right)_{L^2(\omega\times(0,T))} &= \iint_{\omega\times(0,T)} (\alpha\psi_1 + (1-\alpha)\psi_2)v\,dx\,dt \\
&= \iint_Q \psi(w_t-\Delta{w})\,dx\,dt \\
&= \int_\Omega |w(T)|^2(\alpha \mathds{1}_{\mathcal{O}_1} + (1-\alpha)\mathds{1}_{\mathcal{O}_2})\,dx \geq 0.
\end{align*}
These properties of $\Lambda_\alpha$ allow us to conclude that the equation \eqref{pareto-14} admits a unique solution. This ends the proof.
\end{proof}

\section{The semilinear case} \label{sec:Semilinear}

Throughout this section, we will assume that the dynamics of the state is governed by a semilinear heat equation. More precisely, the state equation will be
\begin{equation}\label{pareto-15}
  \begin{cases}
    u_t -\Delta{u} + F(u)= v \mathds{1}_{\omega}, \quad &\mbox{in} \ Q, \\
    u = 0, \quad &\mbox{on} \ \Sigma,\\
    u(0) = u_0,\quad &\mbox{in} \ \Omega,
  \end{cases}
\end{equation}
where $F: \mathds{R} \rightarrow \mathds{R}$ is $C^1$, with a uniformly bounded first-order derivative. As in the linear case, we have:

\begin{lemma}\label{lemma3.1}
For any given $v \in L^2\left( \om \times \left(0,T\right) \right)$ and $u_0 \in L^2(\Om),$ there exists exactly one function $u \in C^0([0,T];L^2(\Om))\cap L^2(0,T;H^1_0(\Om)),$ with $u_t \in L^2(0,T;H^{-1}(\Om)),$ which is a weak solution to \eqref{pareto-15}. Furthermore,
$$
\|u\|_{L^\infty(0,T;L^2(\Om))} + \|u\|_{L^2(0,T;H^1_0(\Om))} + \|u_t\|_{L^2(0,T;H^{-1}(\Om))} \leqslant C\left( \|u_0\| + \|v\|_{L^2\left(\om \times\left(0,T\right)\right)} \right).
$$
\end{lemma}
For a proof, see \cite{evans10}. Now, let us introduce  two uncontrolled trajectories of \eqref{pareto-15} $u_1$ and $u_2$, corresponding to distinct initial data $u_{0i} \in L^2\left(\Omega\right),$ analogously as we did in Section \ref{sec:Linear}. We consider the optimization criteria \eqref{def:ObjFuncs}, with $u^T_i := u_i\left(T\right),$ for our bi-objective problem. We will use the notions of Pareto equilibria and quasi-equilibria given in Definitions \ref{def:ParetoEq} and \ref{def:ParetoQuasiEq}.

Employing standard arguments, we can prove the following differentiability property of $J_1$ and $J_2:$
\begin{lemma} \label{lem:DerivativeJ_Semilinear}
For each $i\in\left\{ 1,2 \right\},$ the functional $J_i$ is G\^ateaux-differentiable and, for any $v\in \mathcal{U}$ and $w \in L^2(\om \times (0,T))$, one has
$$
\left\langle J_i^\prime\left( v \right), w - v\right\rangle = \iint_{ \om \times \left( 0,T\right) } \left( \varphi_i + \mu v \right) (w-v) \,dx\,dt,
$$
where $\varphi_i$ is the solution to
\begin{equation} \label{eq:Adjoint_Semilinear}
    \begin{cases}
  -\varphi_{i,t} - \Delta \varphi_i + F^\prime\left( u \right) \varphi_i = 0, &\text{ in } Q,\\
  \varphi_i = 0, &\text{ on } \Sigma,\\
  \varphi_i(T) =\left( u(T) - u_i(T) \right) \mathds{1}_{\mathcal{O}_i}, &\text{ in } \Om
\end{cases}
\end{equation}
and $u$ solves \eqref{pareto-15}.
\end{lemma}

Henceforth, we will write for brevity
$$
J_{(\alpha)} := \alpha J_1 + \left( 1-\alpha\right) J_2. 
$$
Just as in the linear case, cf. Proposition \ref{prop:ParetoEqToQuasiEq}, we can prove the following.
\begin{proposition}
If $\hat{v}$ is a Pareto equilbrium for the criteria $J_1$ and $J_2,$ then it is also a Pareto quasi-equilibrium.
\end{proposition}

For the converse, one has:
\begin{proposition}
(a) If $\alpha \in \left[0,1\right]$ and $\hat{v}$ minimizes $J_{(\alpha)},$ then $\hat{v}$ is a Pareto quasi-equilibrium.

(b) If $\alpha \in \left(0,1\right)$ and $\hat{v}$ minimizes $J_{(\alpha)},$ then $\hat{v}$ is a Pareto equilibrium.
\end{proposition}


Just as in Section \ref{sec:Linear}, we have the following $\mu = 0$ picture (we omit the proof, which is rather similar to that of the linear case):
\begin{theorem}
    Let us assume that $\mu = 0$ and $\mathcal{U}$ has the continuation property in Theorem \ref{thm:mu0case_linear}.
    \begin{itemize}
        \item If $\alpha = 0$ or $\alpha =1,$ then there exist Pareto equilibria $\hat{v}$ satisfying $J_{(\alpha)}\left( \hat{v} \right) = 0;$
        \item If $0 < \alpha < 1,$ then $J_{\left( \alpha \right)}$ does not admit a minimizer.
    \end{itemize}
\end{theorem}

For the remainder of this section, we focus on the $\mu > 0$ case. Our first result concerns the existence of Pareto quasi-equilibria:
\begin{proposition}\label{th.2}
	Let us assume that $\mu>0.$ Then, for each $\alpha \in [0,1],$ there exists a control $\hat{v}\in L^2\left( \omega\times\left(0,T\right) \right)$ such that $J_{(\alpha)}(\hat{v}) \leqslant J_{(\alpha)}(v),$ for every $v \in L^2\left( \omega\times\left(0,T\right) \right).$ In particular, $\hat{v}$ is a Pareto quasi-equilibrium.
\end{proposition}

The proof follows immediately from the fact that the functional $J_{\left(\alpha\right)} : L^2\left( \om \times \left(0,T\right) \right) \rightarrow \mathds{R}$ is sequentially weakly lower semi-continuous and coercive. It is possible to verify these properties of $J_{\left( \alpha \right)}$ in a standard way. For brevity, we omit the datails.

We now seek to guarantee that the previously identified Pareto quasi-equilibria are in fact Pareto equilibria. If we assume that $F$ is $C^2$ and its first and second order derivatives are uniformly bounded and we restrict ourselves to spatial dimensions not greater than three, this result holds:
\begin{theorem}\label{th.3}
	Let us assume that $N \leqslant 3$, $R>0$ and the set $\mathcal{U}_R := \left\{ v \in \mathcal{U} : \|v\|_{L^2\left( \om \times \left(0,T\right) \right)} \leqslant R \right\}$ is non-empty. There exists $k = k\left( \Om,\, T,\, \|u_{01}\|,\,\|u_{02}\|\right)$ such that, if $\mu > k \left( R + \|u_0\|_{L^2(\Omega)} \right),$ then $J_{\left(\alpha\right)}$ restricted to $\mathcal{U}_R$ is convex. Consequently, if $\hat{v}\in L^2(\om \times (0,T))$ is a Pareto quasi-equilibrium corresponding to some $\alpha\in \left[ 0,1 \right]$ and $\mu$ is sufficiently large, then $\hat{v}$ is a Pareto equilibrium.
\end{theorem}
\begin{proof}
The second part is immediate from the coerciveness of $J_{\left(\alpha\right)}$ for positive $\mu.$ Let us prove that there exists a constant $k > 0,$ depending only on $\Om,\, T,\,F,\, u_1$ and $u_2,$ such that
\begin{equation} \label{eq:2ndOrderDerivativeIsPositive}
     J_{\left(\alpha\right)}^{\prime\prime}\left(v; w - v,w - v\right) \geq \left[ \mu - k\left(1+ R + \|u_0\|_{L^2(\Omega)} \right) \right]\|w - v\|^2_{L^2(\omega\times(0,T))}
\end{equation}
for every $v \in \mathcal{U}_R$ and every $w\in \mathcal{U}.$ After this, if $\mu$ is sufficiently large, namely
$$
\mu > k\left(1+ R + \|u_0\|_{L^2(\Omega)} \right),
$$
it will be clear that $J_{\left( \alpha \right)}$ is convex. 

Let us fix $v \in \mathcal{U}_R$ and let us introduce $\varphi,\, y$ and $\psi,$ with $\varphi := \alpha \varphi^1 + \left( 1-\alpha \right)\varphi^2,$  	
\begin{equation}\label{pareto-20}
  \begin{cases}
    y_{ t} -\Delta y + F^\prime(u)y  = (w - v)\mathds{1}_\mathcal{\omega}, \quad &\mbox{in} \ Q, \\
    y = 0, \quad &\mbox{on} \ \partial\Omega\times(0,T),\\
    y(0) = 0,	\quad &\mbox{in} \ \Omega ,
  \end{cases}
\end{equation}	
and
\begin{equation}\label{pareto-21}
  \begin{cases}
    -{\psi}_{ t} -\Delta{{\psi}} + F'(u)\psi =  -F''(u)\varphi{y}, \quad &\mbox{in} \ Q, \\
    {\psi} = 0, \quad &\mbox{on} \ \partial\Omega\times(0,T),\\
    {\psi}(T) = y(T)\left( \alpha  \mathds{1}_{\mathcal{O}_1} + (1-\alpha)\mathds{1}_{\mathcal{O}_2} \right),	\quad &\mbox{in} \ \Omega .
  \end{cases}
\end{equation}
With these notations, it is standard and not difficult to deduce that $J_{(\alpha)}$ is twice differentiable and
\begin{equation}\label{pareto-22}
	 J_{\left(\alpha\right)}^{\prime\prime}\left(v; w - v,w - v\right) = \iint_{\omega\times(0,T)} \left[ \psi + \mu(w-v) \right](w-v) \ dx\,dt,
\end{equation}
for each $w \in \mathcal{U}.$

From \eqref{pareto-22}-\eqref{pareto-23}, an integration by parts yields
\begin{align}\label{pareto-23}
	\begin{split}
		 \iint_{\omega\times(0,T)} \psi(w-v) \ dx\,dt &= \iint_{Q} \psi[y_t - \Delta{y} + F'(u)y] \ dx\,dt \\
		 &= -\iint_{Q} F^{\prime\prime}(u)\varphi|y|^2 \ dx\,dt + \int_\Omega |y(T)|^2( \alpha \mathds{1}_{\mathcal{O}_1} + (1-\alpha)\mathds{1}_{\mathcal{O}_2}  ) \ dx.
	\end{split}
\end{align}
Since $N \leqslant 3,$ we have the estimates
\begin{equation}\label{pareto-25}
	\|y\|_{L^2(0,T;L^{\infty}(\Omega))} \leqslant C\|v\|_{L^2(\omega\times(0,T))},
\end{equation}
and
\begin{align}\label{pareto-26}
  \begin{split}
	 \|\varphi\|_{L^2(Q)} &\leqslant C\left( 1 + \|u(T)\|_{L^2(\Omega)} \right) \\
	 &\leqslant C\left( 1 + \|v\|_{L^2(\omega\times(0,T))} + \|u_0\|_{L^2(\Omega)}\right),
  \end{split}	
\end{align}
where $C > 0$ depends only on $\Om,\,T,\,F,\,u_1$ and $u_2.$ Employing \eqref{pareto-25}-\eqref{pareto-26}, we can estimate the first integral appearing in the right of \eqref{pareto-23} as follows:
\begin{equation}\label{pareto-27}
  \left| \iint_{Q} F^{\prime\prime}(u)\varphi|y|^2 \ dx\,dt \right| \leq C\left( 1 + R + \|u_0\|_{L^2(\Omega)} \right) \|w - v\|^2_{L^2(\omega\times(0,T))}.
\end{equation}

Therefore, using estimate \eqref{pareto-27} in \eqref{pareto-23}, and recalling \eqref{pareto-22}, we have
$$
 J_{\left(\alpha\right)}^{\prime\prime}\left(v; w - v,w - v\right) \geq (\mu - \mu_0)\|w-v\|^2_{L^2(\omega\times(0,T))},
$$
where $\mu_0$ is of the form $\mu_0 = k(\Omega,\, T,\, F,\,u_1,\,u_2)(1 + R + \|u_0\|_{L^2(\Omega)}).$ This completes the proof.
\end{proof}

\section{The bilinear control case} \label{sec:bilinear}

We devote this section to the analysis of Pareto equilibria of the model:
\begin{equation}\label{pareto-28}
    \begin{cases}
      u_t -\Delta{u} = -uv \mathds{1}_{\omega}, \quad &\mbox{in} \ Q, \\
      u = 0, \quad &\mbox{on} \ \Sigma,\\
      u(0) = u_0,\quad &\mbox{in} \ \Omega.
    \end{cases}
\end{equation}

As before, we will assume that the admissible set $\mathcal{U}$ is a non-empty closed convex set of $L^2(\om \times (0,T))$. We refer to this framework as the bilinear case since the control acts multiplicatively. As in the previous sections, we fix two uncontrolled trajectories $u_1$ and $u_2$ respectively corresponding to distinct initial data, $u_{01}$ and $u_{02}$. Our objective functionals are defined precisely as in \eqref{def:ObjFuncs}. Once again, we consider Pareto equilibria and quasi-equilibria as in Definitions \ref{def:ParetoEq} and \ref{def:ParetoQuasiEq}.

\begin{lemma}
Assume that $N=2$. For every $v\in L^2(\om \times (0,T))$ and every $u_0\in L^2(\Om)$, there exists a unique function $u\in C^0_W([0,T];L^2(\Om))\cap L^2(0,T;H^1_0(\Om))$, with $u_t\in L^2(0,T;W^{-1,\alpha}(\Om))$ for all $\alpha < 2$, which is a weak solution to $(\ref{pareto-28})$. Furthermore, 
$$\|u\|_{L^\infty(0,T;L^2(\Om))}+\|u\|_{L^2(0,T;H^1_0(\Om))}+\|u_t\|_{L^2(0,T;W^{-1,\alpha}(\Om))} \leqslant C(\|u_0\|+\|v\|_{L^2(\om \times (0,T))}).$$\end{lemma}

The proof relies on the fact that $N=2$. It is a little more involved than the proofs Lemmas \ref{lemma2.1} and \ref{lemma3.1}, but can also be achieved with the help of standard estimates (in fact, these estimates will appears soon, in the proof of Theorem \ref{thm:Bilinear1stMainThm}).

Note that a similar result can be established for $N=3$ if we assume that $v \in L^\infty(\om \times (0,T))$. More precisely, one has:

\begin{lemma}
Assume that $N=3$. For every $v\in L^\infty(\om \times (0,T))$ and every $u_0\in L^2(\Om)$, there exists a unique function $u\in C^0([0,T];L^2(\Om))\cap L^2(0,T;H^1_0(\Om))$, with $u_t\in L^2(0,T;H^{-1}(\Om))$, which is a weak solution to $(\ref{pareto-28})$. Furthermore, 
$$\|u\|_{L^\infty(0,T;L^2(\Om))}+\|u\|_{L^2(0,T;H^1_0(\Om))}+\|u_t\|_{L^2(0,T;H^{-1}(\Om))} \leqslant C(\|u_0\|+\|v\|_{L^2(\om \times (0,T))}).$$\end{lemma}
As before, we set 
$$
J_{\left( \alpha \right)} := \alpha J_1 + \left( 1-\alpha \right)J_2,
$$ 
for each $\alpha\in \left( 0,1 \right).$

\begin{theorem}\label{thm:Bilinear1stMainThm}
	We assume that $N=2$ and $\mu > 0.$ For each $\alpha\in \left[ 0,1 \right],$ there exists a control $\hat{v}_\alpha \in L^2(\omega\times(0,T))$ which minimizes the functional $J_{\left( \alpha \right)},$ i.e. that is a Pareto quasi-equilibrium corresponding to $\alpha.$ Moreover, if we restrict ourselves to an $\alpha\in \left( 0,1 \right)$, then $\hat{v}_{\left( \alpha \right)}$ is actually a Pareto equilibrium.
\end{theorem}
\begin{proof}
It is not restrictive to assume that $\mathcal{U}=L^2(\om \times (0,T))$. Similarly as in Theorem \ref{th.2}, we just have to prove that the functional $J_{\left( \alpha \right)}$ is weakly lower semi-continuous
and coercive. The second property is obvious, whence we focus on the first one. Let the $v_n$ satisfy $v_n \rightharpoonup v$ weakly in $L^2(\omega\times(0,T))$ and let $u_n$ be, for each $n$, the state corresponding to the control $v_n,$ i.e.
\begin{equation}\label{pareto-29}
  \begin{cases}
    u_{n,t} -\Delta{u_n} = -v_n u_n \mathds{1}_{\omega}, \quad &\mbox{in} \ Q, \\ 
    u_n = 0, \quad &\mbox{on} \ \Sigma,\\
    u_n(0) = u_0,\quad &\mbox{in} \ \Omega.
  \end{cases}
\end{equation}
By multiplying \eqref{pareto-29} by $u_n$ and integrating over $\Omega,$ we have:
\begin{align*}
    \frac{1}{2}\frac{d}{dt}\|u_n\|_{L^2(\Omega)}^2 + \frac{1}{2}\| \nabla u_n\|_{L^2(\Omega)^2}^2 &= -\int_\om v_n\left| u_n \right|^2 \ dx \\
	&\leqslant C\|v_n\|_{L^2(\omega)}\|u_n\|_{L^4(\Omega)}^2\\
	&\leqslant C\|v_n\|_{L^2(\omega)}\|u_n\|_{L^2(\Omega)}\| \nabla u_n\|_{L^2(\Omega)^2} \\
	&\leqslant \frac{1}{4}\| \nabla u_n\|^2_{L^2(\Omega)} + C\|v_n\|^2_{L^2(\omega)}\|u_n\|^2_{L^2(\Omega)}.
\end{align*}
Integrating the last inequality in time, from $0$ to $t,$ we obtain
\begin{equation}\label{pareto-30}
 \|u_n\|^2_{L^2(\Omega)} + \int_0^t \| \nabla u_n\|^2_{L^2(\Omega)^2} \ dt \leq \|u_0\|^2_{L^2(\Omega)} + C\int_0^t \|v_n\|_{L^2(\omega)}^2\|u_n\|_{L^2(\Omega)}^2 \ dt 
\end{equation}
which, thanks to Gronwall's Lemma, gives
\begin{equation}\label{pareto-31}
	\|u_n\|_{L^{\infty}(0,T;L^2(\Omega))} + \|u_n\|_{L^2(0,T; H_0^1(\Omega))} \leq C.
\end{equation}
It is not difficult to deduce the estimates

\begin{equation}\label{pareto-32}
	\| v_n u_n \mathds{1}_{\omega}\|_{L^2\left( 0,T;L^{1}(\Omega) \right)} \leq C\|v_n\|_{L^2(\omega\times(0,T))}\|u_n\|_{L^\infty(0,T;L^2(\Omega))} \leqslant C.
\end{equation}
From \eqref{pareto-29}, \eqref{pareto-31} and \eqref{pareto-32} we find that, for any $\alpha<2$, 
\begin{equation}\label{pareto-33}
	\|u_{n,t}\|_{L^2(0,T;W^{-1, \alpha}(\Omega))} \leq C.
\end{equation}
Then, from \eqref{pareto-30} and \eqref{pareto-33}, we deduce that the sequence $\left\{ u_n \right\}_{n \geqslant 1}$ is confined to a fixed compact set in $L^2(Q).$ Consequently, at least for a subsequence, we can assume that $u_n$ converges to the state $u$ corresponding to $v$ strongly in $L^2(Q).$  From the previous estimates, we see that the convergence also takes place weakly in $L^2\left(0,T;H^1_0\left(\Om\right) \right)$ and weakly$-*$ in $L^\infty\left( 0,T; L^2\left( \Om \right) \right).$ Furthermore, from \eqref{pareto-33}, we have that $u\in C_W^0([0,T];L^2(\Om))$ and $u_n(T)$ converges to $u(T)$ weakly in $L^2(\Om)$.

From this, we conclude that

$$
\liminf_{n\rightarrow \infty} J_{\left(\alpha\right)}(v_n) \geqslant J_{\left( \alpha\right)}(v).
$$ 
This ends the proof.
\end{proof}

\begin{theorem}\label{th.5}
Let us suppose that $N=2,$ $\mu > 0$ and $\hat{v}$ is a Pareto quasi-equilibrium corresponding to $\alpha \in \left(0,1\right),$ i.e.
$$
\left\langle J_{\left(\alpha\right)}^\prime (\hat{v} ),w-\hat{v} \right\rangle \geqslant 0 \hspace{1.0cm} \forall w \in \mathcal{U}.
$$ 
There exists $k > 0,$ depending only on $\Om,\, T,\, \|\hat{v}\|_{L^2\left( \om \times \left(0,T\right)\right)},\, \left\|u_{01}\right\|$ and $\left\|u_{02}\right\|$ such that, if $\mu > k\|u_0\|^2,$ then $\hat{v}$ is a Pareto equilibrium.
\end{theorem}
\begin{proof}
We only have to prove that $\hat{v}$ minimizes $J_{\left( \alpha \right)}.$ First, note that, since $J_{\left( \alpha \right)}$ is coercive, there exists $R> 0$ such that, if $v\in \mathcal{U}$ and $||v||_{L^2(\om \times (0,T))} > R$, then 
$$
J_{\left(\alpha\right)}(v) \geqslant J_{\left(\alpha\right)}(0).
$$

Let us denote by $B_R$ the closed ball in $L^2(\om \times (0,T))$ of radius $R$. It will suffice to prove that, for any $R>0$, there exists $\mu_0 = \mu_0 (\Om, T, R, ||u_{01}||, ||u_{02}||)$ such that, if $\mu >\mu_0$, then $J_{(\alpha)}$ is convex in $\mathcal{U}_R := \mathcal{U}\cap B_R$. To this purpose, let us check that, for an appropriate $\mu_0$, one has

\begin{equation}\label{pareto-34}
 J_{\left(\alpha\right)}^{\prime\prime}(v;w-v,w-v) \geqslant (\mu - \mu_0)\|w-v\|^2_{L^2(\omega\times(0,T))} \quad \forall v\in \mathcal{U}_R, \quad \forall w\in \mathcal{U}.
\end{equation}

If we denote by $z$ the G\^ateaux derivative of the state associated to the control $v$ in the direction of $w-v,$ we have:
$$
\begin{cases}
z_t -\Delta{z} = -(w-v) u \mathds{1}_{\omega} -v z \mathds{1}_{\omega}, \quad &\mbox{in} \ Q , \\
z = 0, \quad &\mbox{on} \ \Sigma,\\
z(0) = 0,\quad &\mbox{in} \ \Omega.
\end{cases}
$$
Carrying out standard arguments, we see that $J_{\left(\alpha \right)}$ is differentiable, with
$$
\left\langle J_{\left(\alpha\right)}^\prime\left(v\right),w-v\right\rangle = \iint_{\om\times\left(0,T\right)} \left(\varphi u + \mu v\right)(w-v)\,dx\,dt,
$$
where the adjoint state $\varphi$ satisfies
\begin{equation} \label{eq:Bilinear_AdjointPDEs}
    \begin{cases}
-{\varphi}_{ t} -\Delta{{\varphi}}  = -v \varphi \mathds{1}_{\omega}, \quad &\mbox{in} \ Q \\
{\varphi} = 0, \quad &\mbox{on} \ \Sigma,\\
{\varphi}(T) = \alpha( u(T) - u_1(T))\mathds{1}_{\mathcal{O}_1} + (1-\alpha)( u(T) - u_2(T))\mathds{1}_{\mathcal{O}_2} ,	\quad &\mbox{in} \ \Omega.
\end{cases}
\end{equation}
Therefore, $J_{(\alpha)}$ is twice differentiable and the following identities hold (under self-explanatory notations):
\begin{align}\label{pareto-38}
  \begin{split}
	J_{\left(\alpha\right)}^{\prime\prime}(v;w-v,w-v) \\
	&:= \lim_{\epsilon \to 0} \frac{1}{\epsilon}\left( \left\langle J_{\left(\alpha\right)}^\prime(v+\epsilon (w-v)),w-v\right\rangle -\left\langle J_{\left(\alpha\right)}^\prime(v),w-v\right\rangle \right) \\
	 &= \lim_{\epsilon \to 0} \frac{1}{\epsilon} \iint_{\omega\times(0,T)}\left\{ \left[ \varphi_\epsilon{u_\epsilon} +\mu(v+\epsilon(w-v)) \right] - ( \varphi{u} +\mu{v} ) \right\}(w-v) \ dx\,dt \\
	 &= \lim_{\epsilon \to 0} \frac{1}{\epsilon}\iint_{\omega\times(0,T)} [\varphi_\epsilon(u_\epsilon-u) + (\varphi_\epsilon-\varphi)u](w-v) \ dxdt + \mu\iint_{\omega\times(0,T)} |w-v|^2 \ dx\,dt.
  \end{split}
\end{align}
Let us introduce $z_\epsilon := \epsilon^{-1}\left( u_\epsilon - u\right)$ and $\psi_\epsilon := \epsilon^{-1}\left(\varphi_\epsilon -\varphi\right).$ Then, $z_\epsilon$ and $\psi_\epsilon$ respectively satisfy
\begin{equation}\label{pareto-39}
  \begin{cases}
    z_{\epsilon ,t} -\Delta{z_\epsilon} = - [ v+\epsilon{(w-v)} ]{z_\epsilon}\mathds{1}_{\omega} - (w-v) u \mathds{1}_{\omega} , \quad &\mbox{in} \ Q , \\
    z_\epsilon = 0, \quad &\mbox{on} \ \Sigma ,\\
    z_\epsilon(0) = 0,\quad &\mbox{in} \ \Omega
  \end{cases}
\end{equation}
and
\begin{equation}\label{pareto-40}
  \begin{cases}
    -{\psi}_{\epsilon ,t} -\Delta{{\psi_\epsilon}} =  - (v+\epsilon{(w-v)}) \psi_\epsilon \mathds{1}_{\omega} - (w-v) \varphi \mathds{1}_{\omega}, \quad &\mbox{in} \ Q, \\
    {\psi_\epsilon} = 0, \quad &\mbox{on} \ \Sigma,\\
    {\psi_\epsilon}(T) =  \alpha z_\epsilon(T)\mathds{1}_{\mathcal{O}_1} + (1-\alpha)z_\epsilon(T)\mathds{1}_{\mathcal{O}_2},	\quad &\mbox{in} \ \Omega.
  \end{cases}
\end{equation}
Furthermore, $z_\epsilon \to z$ and $\psi_\epsilon \to \psi$ strongly in $L^2\left(Q\right),$ where $z$ and $\psi$ solve the coupled PDE system
\begin{equation}\label{pareto-41}
  \begin{cases}
    z_t -\Delta{z} = -v z\mathds{1}_{\omega} - (w-v) u\mathds{1}_{\omega}, \quad &\mbox{in} \ Q, \\
    -\psi_t -\Delta{\psi} = -v \psi \mathds{1}_{\omega} - (w-v) \varphi \mathds{1}_{\omega}, \quad &\mbox{in} \ Q, \\
    z = 0 \text{ and } \psi = 0, \quad &\mbox{on} \ \Sigma,\\
    z(0) = 0 \text{ and } {\psi}(T) =  (\alpha \mathds{1}_{\mathcal{O}_1} + (1-\alpha)\mathds{1}_{\mathcal{O}_2})z(T), \quad &\mbox{in} \ \Omega .
  \end{cases}
\end{equation}
Consequently, from \eqref{pareto-38}, the following is found:
\begin{equation}\label{pareto-42}
	J_{\left(\alpha\right)}^{\prime\prime} (v; w-v,w-v) = \iint_{\omega\times(0,T)} (\varphi{z} + \psi{u})(w-v) \ dx\,dt + \mu\iint_{\omega\times(0,T)} |w-v|^2 \ dx\,dt.
\end{equation}
We also observe from \eqref{pareto-41} that
\begin{equation}\label{pareto-43}
\iint_{\omega\times(0,T)} \varphi z (w-v)  \ dx\,dt = \iint_{\omega\times(0,T)} \psi  u(w-v) \, dx\,dt + \int_\Omega (\alpha \mathds{1}_{\mathcal{O}_1} + (1-\alpha)\mathds{1}_{\mathcal{O}_2})|z(T)|^2\,dx,
\end{equation}
whence
\begin{equation}\label{pareto-44}
	 J_{\left(\alpha\right)}^{\prime\prime}(v;w-v,w-v) \geq 2\iint_{\omega\times(0,T)} \psi u (w-v) \ dx\,dt + \mu\iint_{\omega\times(0,T)} |w-v|^2 \, dx\,dt.
\end{equation}
Moreover, we have the estimate
$$
\left| \int_{\om \times \left(0,T\right)} \psi (w-v) u \,dx\,dt \right| \leqslant \|\psi\|_{L^4(Q)}\|w-v\|_{L^2\left(\om\times\left(0,T\right)\right)}\|u\|_{L^4(Q)}. 
$$
Arguing as in the proof of Theorem \ref{thm:Bilinear1stMainThm}, we also find that
$$
\|u\|_{L^4(Q)} + \|\varphi\|_{L^4(Q)} \leqslant C\left( \Om,\,T,\,R,\,\|u_{01}\|,\,\|u_{02}\| \right)\|u_0\|
$$
and
$$
\|\psi\|_{L^4(Q)} \leqslant C\left( \Om,\,T,\,R,\,\|u_{01}\|,\,\|u_{02}\| \right)\|u_0\|\|w-v\|_{L^2\left( \om \times \left(0,T\right) \right)}.
$$
Consequently,
\begin{align}
    \begin{split}
        J_{\left(\alpha\right)}^{\prime\prime} (v;w-v,w-v) &\geqslant \mu \iint_{\om \times \left(0,T\right)}|w-v|^2\,dx\,dt - 2\|\psi\|_{L^4(Q)}\|w-v\|_{L^2\left(\om\times\left(0,T\right)\right)}\|u\|_{L^4(Q)} \\
        &\geqslant \left[ \mu - C\left( \Om,\,T,\,R,\,\|u_{01}\|,\,\|u_{02}\| \right)\|u_0\|^2 \right] \|w-v\|_{L^2\left( \om \times \left(0,T\right) \right)}^2.
    \end{split}
\end{align}
This ends the proof.
\end{proof}

As a consequence of the proof of Theorem \ref{thm:Bilinear1stMainThm}, we obtain the following result.
\begin{corollary}
Let us assume that $N=2.$ For each $R>0,$ there exists $\mu_0 > 0$ such that, if $\mu > \mu_0,$ then the restriction of the functional $J_{\left(\alpha\right)}$ to the set
$$
\mathcal{U}_R = \left\{ v \in \mathcal{U} : \|v\|_{L^2\left( \om \times \left(0,T\right)\right)} \leqslant R \right\}
$$
is $(\mu-\mu_0)-$convex.
\end{corollary}

The proof of the previous results has the dimensionality assumption $N=2$ as an essential restriction. Indeed, the key embedding $$W^{-1,\alpha}(\Om)\hookrightarrow L^1(\Om)$$ employed in \ref{pareto-33} does not hold in the three-dimensional setting. However, if we impose to the admissible control set to be a subset of $L^\infty(\om\times (0,T))$, we can obtain results which are still valid for $N=3$.

\begin{theorem}
Let us assume that $\mathcal{U} \subseteq \left\{ v \in L^2(\om \times (0,T)):\ |v|\leqslant R\  \text{a.e. in}\ Q \right\}$ for some $R>0.$ Then:

(a) For each $\alpha \in \left[0,1\right],$ the functional $J_{\left(\alpha\right)}$ attains its minimum in $\mathcal{U}$.

(b) If $N\leqslant 3,$ $u_0,u_{01},u_{02} \in H^1_0(\Om) $ and $\mu$ is sufficiently large, any Pareto quasi-equilibrium $\hat{v}$ corresponding to an $\alpha \in \left(0,1\right)$ is a Pareto equilibrium.
\end{theorem}

\begin{proof}
($a$) It suffices to check that $J_{\left(\alpha\right)}$ is sequentially weakly lower semicontinuous. Let us assume that $v^n \in \mathcal{U}$ for all $n \geqslant 1$ and $v^n \rightharpoonup v$ weakly in $L^2\left( \om \times \left( 0,T \right) \right).$ Let $u^n$ denote the solution corresponding to $v^n.$ It is then clear that 
$$
\frac{1}{2}\frac{d}{dt}\left[ \int_\Om \left|u^n\right|^2\,dx \right] + \int_\Om \left| \nabla u^n \right|^2\,dx \leqslant R \int_\Om \left| u^n\right|^2\,dx
$$
and, from Gronwall's Lemma, we get
$$
\sup_{\left[0,T\right]}\|u^n\|^2 + \|\nabla u^n\|^2 \leqslant C \|u_0\|^2,
$$
for a constant $C$ depending only on $\Om,\,T$ and $R.$ Now, standard arguments allow to deduce that, at least for a subsequence, $u^n(T)$ converges to $u(T)$ weakly in $L^2(\Om),$ where $u$  is the state corresponding to $v.$ Therefore, we conclude that $J_{\left( \alpha \right)}(v) \leqslant \liminf_{n \rightarrow \infty} J_{\left( \alpha \right)}(v^n),$ as desired.

($b$) Proceeding as in Theorem \ref{th.5}, see \eqref{pareto-44}, we have:
\begin{equation} \label{eq:SecondDeriv_UnifBddControls}
     J_{\left(\alpha\right)}^{\prime\prime} ( v; w-v,w-v) \geq 2\iint_{\omega\times(0,T)} \psi u (w-v)\ dx\,dt + \mu\iint_{\omega\times(0,T)} |w-v|^2 \, dx\,dt,
\end{equation}
where $u$ is the state corresponding to $v$ and $(z,\psi)$ solves \eqref{pareto-41}. Next, we observe that
\begin{equation} \label{eq:Est1BilinearUnifBdd}
    \sup_{\left[0,T\right]}\left(  \|\nabla u\| + \|\nabla \varphi\| \right)\leqslant C\left( \Om,\, T,\, R,\, \|u_{01}\|_{H^1_0(\Om)},\,\|u_{02}\|_{H^1_0(\Om)} \right)\|u_0\|_{H^1_0(\Om)}.
\end{equation}
Moreover, we can also infer from \eqref{pareto-41} the following inequality
\begin{equation} \label{eq:Est2BilinearUnifBdd}
    \sup_{\left[0,T\right]}\|\nabla \psi\| \leqslant C\left( \Om,\, T,\, R,\, \|u_{01}\|_{H^1_0(\Om)},\,\|u_{02}\|_{H^1_0(\Om)} \right)\|u_0\|_{H^1_0(\Om)}\|w-v\|_{L^2\left( \om \times \left(0,T\right) \right)}.
\end{equation}
For brevity, let us denote $L^\infty(0,T;H^1_0(\Om))$ by $W.$ With the help of the embedding $H^1(\Om) \hookrightarrow L^6(\Om)$ (valid for $N\leqslant 3$), we conclude from \eqref{eq:SecondDeriv_UnifBddControls}, \eqref{eq:Est1BilinearUnifBdd} and \eqref{eq:Est2BilinearUnifBdd} what follows:
\begin{align*}
     J_{\left(\alpha\right)}^{\prime\prime} (v;w-v,w-v) &\geqslant \mu \|w-v\|^2_{L^2\left( \om \times \left(0,T\right) \right)} - \|w-v\|_{L^2\left( \om \times \left(0,T\right) \right)} \|u\|_{L^6(Q)}\|\psi\|_{L^3(Q)} \\
    &\geqslant \mu \|w-v\|^2_{L^2\left( \om \times \left(0,T\right) \right)} - C\|w-v\|_{L^2\left( \om \times \left(0,T\right) \right)} \|u\|_{W}\|\psi\|_{W} \\
    &\geqslant \left( \mu - C\right) ||w-v||^2_{L^2(\om \times (0,T))}.
\end{align*}
This ends the proof.
\end{proof}

\section{Computation of equilibria (I): Algorithms}\label{sec:Numerics1}

\subsection{Linear problem} \label{subsec:Linear_Algos}
We refer here to the computation of Pareto equilibria for (\ref{def:ObjFuncs}) and (\ref{pareto-1}). Recall that the associated optimality systems are (\ref{pareto-1}), (\ref{pareto-5}), (\ref{paretoeq_linear_mupos}) for $\alpha \in [0,1]$.
\subsubsection{The Conjugate Gradient Algorithm (CGA)} \label{subsubsec:Linear_ConjGrad}

We first recall an algorithm intended to solve the problem
\begin{equation}\label{eq:LinearVariationalProblem}
a(u,v)=L(v), \quad \forall v\in V,\ u\in V,
\end{equation}
where $V$ is a Hilbert space with norm and scalar product respectively denoted by $||\cdot||_V$ and $(\cdot,\cdot)_V$, $a:V\times V \to \mathbb{R}$ is a symmetric continuous and coercive bilinear form and $L:V\to \mathbb{R}$ is a continuous linear form. The classical Lax-Milgram Lemma ensures that \eqref{eq:LinearVariationalProblem} has a unique solution. It can be found via the so called (optimal step) conjugate gradient algorithm:
\begin{algorithm} \label{algo:ConjugateGradientLinearProblem}
\SetAlgoLined
\KwResult{The solution $u$ of problem \eqref{eq:LinearVariationalProblem}.}
Initialize by providing the tolerance $\epsilon_0,$ the error variable $\epsilon,$ $k=0,$ an initial guess for the state variable $u^0 \in V,$ the solution $g^0$ to the problem
$$
\left( g^0,v \right)_V=a(u^0,v)-L(v) \quad  \forall v \in V,\ g^0\in V
$$
and $z^0 := g^0.$\\
\While{$\epsilon > \epsilon_0,$}{
$1.$ Set $\rho^k := \|g^k\|^2/a(z^k,z^k);$\\
$2.$ Update $u^{k+1} \gets u^k-\rho^k z^k;$\\
$3.$ Compute the solution $g^{k+1}$ to the problem
$$
\left( g^{k+1},v \right)_V=a(g^k,v)-\rho^k(z^k,v), \quad \forall v \in V,\ g^{k+1}\in V;
$$\\
$4.$ Set $z^k \gets g^{k+1}+ \frac{\|g^{k+1}\|^2}{\|g^k\|^2} z^k;$\\
$5.$ Update $\epsilon;$\\
$6.$ $k \gets k+1.$
}
\caption{The CGA for the linear problem}
\end{algorithm}



Concerning the convergence of Algorithm $1$, we have that
$$
\lim_{k \to \infty}\|u^k-u \|=0,
$$
where $u$ is the solution to problem \eqref{eq:LinearVariationalProblem}. In fact, it can be shown that
$$
\|u^k-u\|\leqslant C\|u^0-u\|\left[\frac{\sqrt{c_A}-1}{\sqrt{c_A}+1}\right]^k$$
(see \cite{daniel1971approximate}), where $c_A$ is the condition number of the operator $A$ defined by  $\left(Au, v\right)_V = a(u,v)$ for all $u,v\in V$, that is, 
$$
c_A=\|A\|_{\mathcal{L}(V)}\cdot\|A^{-1}\|_{\mathcal{L}(V)}.
$$ 
We can use the steps Algorithm $1$ to produce Algorithm $2$ (see below) which allows to compute Pareto optimal strategies in the linear case. As for Algorithm $1,$ it can be proved that
$$
\|v^k - v\| \leqslant C\|v^0-v\|\lambda^k,
$$
for a suitable $0 < \lambda < 1.$
\begin{algorithm}
\SetAlgoLined
\KwResult{The Pareto optimal control $v.$}
Initialize by inputting the tolerance $\epsilon_0,$ the error variable $\epsilon,$ $k=0,$ an initial guess $v^0 \in V$ for the Pareto equilibrium, the solutions $u^0,\, \varphi^0_1$ and $\varphi^0_2$ to the problems
\begin{equation*}
\begin{cases}
  {u}^0_t -\Delta {u}^0= v^0\mathds{1}_\omega,  & \text{in}\ Q, \\
  {u}^0 = 0, &\text{on} \ \Sigma,\\
  {u}^0(0) =u_0 ,	&\text{in} \ \Omega 
\end{cases}
\end{equation*}
and
\begin{equation*}
\begin{cases}
 -{\varphi}^0_{i,t} -\Delta{{\varphi}^0_i}= 0, & \text{in}\ Q, \\
  {\varphi}^0_i = 0, &\text{on} \ \Sigma,\\
  {\varphi}^0_i(T) = ({u}^0(T)-u_i(T))\mathds{1}_{\mathcal{O}_i}, 	&\text{in} \ \Omega 
\end{cases}
\end{equation*}
and set $g^0 :=\alpha\varphi^0_1\mathds{1}_\omega+(1-\alpha)\varphi^0_2\mathds{1}_\omega+\mu v^0$ and $z^0:=g^0.$\\
\While{$\epsilon > \epsilon_0,$}{
$1.$ Solve, via Algorithm $1,$ the problems:
\begin{equation*}
\begin{cases}
  w^k_t -\Delta w^k= z^k\mathds{1}_\omega , & \text{in}\ \Omega\times(0,T), \\
  w^k = 0, &\text{on} \ \partial\Omega\times(0,T),\\
  w^k(0) =0, 	&\text{in} \ \Omega ;
\end{cases}
\end{equation*}
\begin{equation*}
\begin{cases}
 -\psi^k_{i,t} -\Delta{\psi^k_i}= 0, & \text{in}\ Q, \\
  \psi^k_i = 0, &\text{on} \ \Sigma,\\
  \psi^k_i(T) = w^k(T)\mathds{1}_{\mathcal{O}_i}, 	&\text{in} \ \Omega,\,\text{ for } i \in \left\{1,2\right\}; 
\end{cases}
\end{equation*}\\
$2.$ Set $\bar{g}^k=\alpha\psi^k_1\mathds{1}_\omega+(1-\alpha)\psi^k_2\mathds{1}_\omega+\mu z^k;$\\
$3.$ Put $\rho^k= \|g^k\|_V^2 / \iint_{\omega\times (0,T)} \left( \alpha\psi^k_1z^k+(1-\alpha)\psi^k_2z^k+\mu (z^k)^2 \right) \, dx\,dt;$\\
$4.$ Update $v:= v^{k+1} \gets v^k-\rho^kz^k;$\\
$5.$ Update $g^{k+1}\gets g^k-\rho^k\bar{g}^k;$\\
$6.$ Update $z^{k+1}\gets g^{k+1}+\frac{\|g^{k+1}\|^2}{\|g^k\|^2}z^k;$\\
$7.$ Update $\epsilon;$\\
$8.$ Update $k\gets k+1.$
}
\caption{The CGA to compute a Pareto equilibrium in the linear model}
\end{algorithm}

\subsection{Semilinear problem} \label{subsec:Semilinear_Algos}
In this section, the state equation is (\ref{pareto-15}) and the cost functions are again given by (\ref{def:ObjFuncs}). Recall that the optimality system are in this case given by (\ref{pareto-15}), (\ref{eq:Adjoint_Semilinear}) and (\ref{paretoeq_linear_mupos}) for $\alpha \in [0,1]$.
\subsubsection{A Fixed Point Algorithm} \label{subsubsec:Semilinear_FixedPoint}




\begin{lemma} \label{lem:ConvSemilinearAlgo}
Assume that $v, v^\prime$ $\in L^2(\omega\times (0,T)),$ $\|v\|_{L^2(\omega\times (0,T))} \leqslant R$ and $\|v^\prime\|_{L^2(\omega\times (0,T))}\leqslant R.$ Let us denote by $u$ and $u^{\prime}$ the solutions to \eqref{pareto-15} corresponding to $v$ and $v^\prime.$ Also, denote by $\varphi$ and $\varphi^{\prime}$ as the corresponding solutions to the adjoint systems \eqref{eq:Adjoint_Semilinear}. Then, there exists a constant $C = C(\Om,\, T,\, F,\, R)$ such that
\begin{align*}
    \|u-u^{\prime}&\|_{L^\infty(0,T;L^2(\Omega))}+\|u-u^{\prime}\|_{L^2(0,T;H^1_0(\Omega))} \\
    &+\|\varphi-\varphi^{\prime}\|_{L^\infty(0,T;L^2(\Omega))}+\|\varphi-\varphi^{\prime}\|_{L^2(0,T;H^1_0(\Omega))} \\
    &\leqslant C\|v-v^\prime\|_{L^2(\omega\times (0,T))}
\end{align*}
\end{lemma}

\begin{proof}
We observe that $u-u^{\prime}$ satisfies 
\begin{equation} \label{eq:uvMinusUvprimePDE}
  \begin{cases}
    (u-u^{\prime})_t -\Delta(u-u^{\prime})+ F(u)-F(u^{\prime})= (v-v^\prime) \mathds{1}_{\omega}, \quad &\mbox{in} \ Q, \\
    u-u^{\prime}= 0, \quad &\mbox{on} \ \Sigma,\\
    u-u^{\prime}(0) = 0,\quad &\mbox{in} \ \Omega,
  \end{cases}
\end{equation}
Multiplying the PDE \eqref{eq:uvMinusUvprimePDE} by $u - u^{\prime}$ and integrating over $\Om,$ we deduce 
\begin{align*}
\frac{1}{2}\frac{d}{dt} \|u-u^{\prime}\|^2+\|\nabla(u-u^{\prime})\|^2 &= -\int_{\Omega} (F(u)-F(u^{\prime}))(u-u^{\prime}) \, dx+ \int_{\Omega} (v-v^\prime)(u-u^{\prime}) \,dx \\
&\leqslant \int_{\Omega} |F(u)-F(u^{\prime})||u-u^{\prime}| \, dx + \int_{\Omega} |v-v^\prime||u-u^{\prime}| \,dx \\
&\leqslant C\|u-u^{\prime}\|^2 +C\|v-v^\prime\|^2. 
\end{align*}
Therefore,
$$
\frac{d}{dt} \|u-u^{\prime}\|^2+\|\nabla(u-u^{\prime})\|^2 \leqslant C\|u-u^{\prime}\|^2+\|v-v^\prime\|^2. 
$$
whence Gronwall's Lemma implies
$$
\|u-u^{\prime}\|_{L^\infty(0,T;L^2(\Omega))}^2+\|u-u^{\prime}\|_{L^2(0,T;H^1_0(\Omega))}^2\leqslant C\|v-v^\prime\|_{L^2\left( \om \times \left(0,T\right) \right)}^2.
$$
This establishes the estimate we asserted for $u - u^{\prime}.$ Now, we analyze difference of the adjoint variables $\varphi - \varphi^{\prime}.$ Proceeding as before, we see that
\begin{align*}
-\frac{1}{2}\frac{d}{dt} \|\varphi-\varphi^{\prime}\|^2+\|\nabla(\varphi-\varphi^{\prime})\|^2 &= -\int_{\Omega} (F^\prime(u)\varphi-F^\prime(u^{\prime})\varphi^{\prime})(\varphi-\varphi^{\prime}) \, dx \\
&\leqslant\int_{\Omega} |F^\prime(u)\varphi-F^\prime(u^{\prime})\varphi^{\prime}|\cdot|\varphi-\varphi^{\prime}| \, dx  \\
&\leqslant C\int_{\Omega}|F^\prime(u)\|\varphi-\varphi^{\prime}|^2\, dx \\
&\hspace{0.5cm}+ C\int_\Om |\varphi^{\prime}|\cdot|F^\prime(u)-F^\prime(u^{\prime})|\cdot|\varphi-\varphi^{\prime}| \, dx  \\
&\leqslant C (\|u-u^{\prime}\|^2+\|\varphi-\varphi^{\prime}\|^2)  \\
&\leqslant C(\|\varphi-\varphi^{\prime}\|^2+\|v-v^\prime\|^2).
\end{align*}
Employing Gronwall's Lemma once more, we conclude that
$$
\|\varphi-\varphi^{\prime}\|_{L^\infty(0,T;L^2(\Omega))}^2+\|\varphi-\varphi^{\prime}\|_{L^2(0,T;H^1_0(\Omega))}^2\leqslant C\|v-v^\prime\|^2.
$$
\end{proof}

This lemma suggests the following iterates, that correspond to a fixed-point formulation of the optimality system:

\begin{algorithm}
\SetAlgoLined
\KwResult{The Pareto optimal control $v,$ and the corresponding state variable $u.$}
Initialize with the error variable $\epsilon,$ the tolerance $\epsilon_0,$ $k=0,$ an initial guess $v^0$ for the control, the corresponding solution $u^0$ of \eqref{pareto-15}, and the adjoint states $\varphi_1^0$ and $\varphi_2^0$ of \eqref{eq:Adjoint_Semilinear}.\\
\While{$\epsilon > \epsilon_0$}{
$1.$ Update $v^{k+1} \gets P_{\mathcal{U}}\left( -\frac{1}{\mu}\left[ \alpha\varphi_1^k + \left(1-\alpha\right)\varphi_2^k \right] \right);$ \\
$2.$ Let $u^{k+1}$ be the solution of \eqref{pareto-15} corresponding to $v=v^{k+1};$\\
$3.$ Let $\varphi_i^{k+1}$ solve \eqref{eq:Adjoint_Semilinear}, with $u = u^{k+1};$\\
$4.$ Update $\epsilon;$\\
$5.$ $k \gets k+1.$
}
\caption{Fixed point iterative algorithm for the semilinear model}
\end{algorithm}

\begin{theorem}
If $\mu$ is sufficiently large, then the states $u^k$ and the controls $v^k$ computed via Algorithm $3$ converge in the following sense:
$$
u^k \rightarrow u \text{ strongly in } L^\infty(0,T; L^2(\Om)) \cap L^2\left( 0,T; H^1_0(\Om) \right),
$$
and
$$
v^k \rightarrow v \text{ strongly in } L^2\left( \om \times \left(0,T\right) \right).
$$
If $\mathcal{U}$ has the property that $P_\mathcal{U}(f^k) \rightarrow P_\mathcal{U}(f)$ a.e. in $Q$ whenever $f^k \rightarrow f$ a.e. in $Q,$ then the convergences above hold a.e. in $Q.$ 
\end{theorem}
\begin{proof}
From Lemma \ref{lem:ConvSemilinearAlgo}, we have
\begin{align*}
    \|\varphi^{k+1} - \varphi^k\|_{L^\infty(0,T; L^2(\Om))} + \|\varphi^{k+1} - \varphi^k\|_{L^\infty(0,T; H^1_0(\Om))} &\leqslant \frac{C}{\mu}\|\varphi^k - \varphi^{k-1}\|_{L^2\left(\om \times \left(0,T\right) \right)} \\
    &\leqslant \left(\frac{C}{\mu}\right)^k\|\varphi^1 - \varphi^0\|_{L^2\left(\om \times \left(0,T\right) \right)},
\end{align*}
whence $v^k = P_\mathcal{U}\left( - \varphi^k / \mu \right) \rightarrow v$ in $L^2(\om\times (0,T))$, where $v:=P_\mathcal{U}\left( - \varphi / \mu \right)$. Since
$$
\sum_{k=1}^\infty \left( \|\varphi^{k+1} - \varphi^k\|_{L^\infty(0,T; L^2(\Om))} + \|\varphi^{k+1} - \varphi^k\|_{L^\infty(0,T; H^1_0(\Om))} \right) < +\infty,
$$
the almost everywhere convergence $v^k \rightarrow v$ also holds. From this, the convergence $u^k$ towards $u$ follows (in the appropriate topology of the statement).
\end{proof}

\subsubsection{A Newton-Raphson Algorithm}

We can also prove the convergence of a Newton-Raphson algorithm, for arbitrary $\mu>0,$ as long as the initial data of the controlled state is close enough to those of the trajectories. We introduce the mapping $\Phi : Y \mapsto Z$, with
$$
\Phi(u,\varphi) := \left( u_t -\Delta u +F(u) -\frac{1}{\mu}\varphi, -\varphi_{t} - \Delta \varphi + F^\prime(u)\varphi,  u\left(\cdot, 0\right), \varphi\left(\cdot,T\right) - u\left(\cdot,T\right)\right),
$$
where $Y=Y_0 \times Y_0$,
$$
Y_0 := \left\{ f\in L^\infty\left( 0,T; H^1_0\left( \Om \right) \right)\cap L^2\left( 0,T; H^2\left( \Om \right) \right)  : f_t \in L^2(Q) \right\}
$$
and 
$$
Z:= \left[ L^2\left(Q\right) \right]^2 \times \left[ H^1_0\left(\Om\right) \right]^2.
$$
We assume that the spaces $Y$ and $Z$ are endowed with the natural topologies (becoming consequently Banach spaces). We remark that $\Phi$ is well-defined. Furthermore, this mapping has the following properties:
\begin{lemma}
The mapping $\Phi : Y \mapsto Z$ is of class $C^1.$ Moreover, for any $u^*\in Y_0$, the linear mapping $\Phi^\prime(u^*,0) : Y \mapsto Z$ is bijective.
\end{lemma}
\begin{proof}
It is straightforward to show that the G\^ateux derivative of $\Phi$ at any $(u,\varphi)\in Y$ in the direction $(y,\psi)$ is given by
$$
\left\langle \Phi^\prime(u,\varphi),\, (y,\psi)\right\rangle = \left(y_t - \Delta y + F^\prime(u)y - \frac{\psi}{\mu},\, -\psi_t -\Delta \psi + F^\prime(u)\psi + F^{\prime\prime}(u)y\varphi,\,y(0),\, \psi(T) - y(T) \right).
$$
We will show that, for all $(u,\varphi),\ (\bar{u},\bar{\varphi}),\ (y,\psi)\in Y$, one has
\begin{equation} \label{eq:NewtonRaphson_ProofOfC1}
    \| \left\langle \Phi^\prime(u,\varphi) - \Phi^\prime(\overline{u},\overline{\varphi}) , (y,\psi)\right\rangle \|_Z \leqslant C(1+ \|( \overline{u},\, \overline{\varphi})\|_Y) \|(u-\overline{u},\, \varphi - \overline{\varphi})\|_Y \|(y,\psi)\|_Y, 
\end{equation}
 whence we will get in particular that $\Phi$ is of class $C^1$. We have:
\begin{equation} \label{eq:NR_Lemma_step1}
    \left\langle \Phi^\prime(u,\varphi) - \Phi^\prime(\overline{u},\overline{\varphi}) , (y,\psi)\right\rangle = (\phi_1,\phi_2,0,0),
\end{equation}
where
$$
\begin{cases}
\phi_1 = \left[ F^\prime(u) - F^\prime(\overline{u}) \right]y,\\
\phi_2 = \left[ F^\prime(u) - F^\prime(\overline{u}) \right]\psi + \left[F^{\prime\prime}(u)\varphi - F^{\prime\prime}(\overline{u})\overline{\varphi} \right]y.
\end{cases}
$$
It is easy to see that
\begin{align} \label{eq:NR_Lemma_step2}
  \begin{split}
    \|\phi_1\|_{L^2(Q)}^2 &\leqslant C\|u-\overline{u}\|_{L^4(Q)}\|y\|_{L^4(Q)} \\
    &\leqslant C\|u-\overline{u}\|_{L^\infty(0,T;H^1_0(\Om))}\|y\|_{L^\infty(0,T;H^1_0(\Om))} \\
    &\leqslant C\|( u-\overline{u},\,\varphi - \overline{\varphi})\|_{Y}\|( y,\,\psi)\|_{Y}.
  \end{split}
\end{align}
Similarly,
$$
\|\left[ F^\prime(u) - F^\prime(\overline{u}) \right]\psi\|_{L^2(Q)} \leqslant C \|(u-\overline{u},\varphi-\overline{\varphi})\|_Y \| (y,\psi) \|_Y
$$
and, writing that
$$
\left[F^{\prime\prime}(u)\varphi - F^{\prime\prime}(\overline{u})\overline{\varphi} \right]y = F^{\prime\prime}(u)(\varphi - \overline{\varphi})y +\left[F^{\prime\prime}(u)- F^{\prime\prime}(\overline{u})\right]\overline{\varphi} y =: \phi_2^1 + \phi_2^2,
$$
we likewise estimate
\begin{equation} \label{eq:NR_Lemma_step3}
    \|\phi_2^1\|_{L^2(Q)} \leqslant C \|(u-\overline{u},\varphi-\overline{\varphi})\|_Y \| (y,\psi) \|_Y.
\end{equation}
We also have
\begin{align*}
    \int_\Om |\phi_2^2|^2\,dx &\leqslant C \int_\Om |u-\overline{u}|^2|\overline{\varphi}|^2|y|^2\,dx \\
    &\leqslant C \sup_\Om |y|^2 \left( \int_\Om |u-\overline{u}|^4\,dx \right)^{1/2} \left( \int_\Om \left|\overline{\varphi} \right|^4 \,dx\right)^{1/2} \\
    &\leqslant C \|\Delta y\|^2 \|\nabla u-\nabla \overline{u}\|^2 \|\nabla \overline{\varphi}\|^2,
\end{align*}
whence
\begin{align} \label{eq:NR_Lemma_step4}
  \begin{split}
    \|\phi_2^2\|_{L^2(Q)} &\leqslant C \|\Delta y\|_{L^2(Q)} \| u- \overline{u}\|_{L^\infty(0,T;H^1_0(\Om))} \| \overline{\varphi}\|_{L^\infty(0,T;H^1_0(\Om))} \\
    &\leqslant C \|\left(\overline{u},\overline{\varphi}\right)\|_Y\|(u-\overline{u},\varphi-\overline{\varphi})\|_Y \| (y,\psi) \|_Y.
  \end{split}
\end{align}
Putting together \eqref{eq:NR_Lemma_step1}-\eqref{eq:NR_Lemma_step4}, we prove \eqref{eq:NewtonRaphson_ProofOfC1}.

Given $u^*_0 \in Y,$ and $(g,h,y_0,\psi_T) \in Z,$ to solve the equation $\Phi^\prime(u^*,0)\cdot(y,\psi) = (g,h,y_0,\psi_T)$ is equivalent to find $(y,\psi) \in Y$ such that
$$
\begin{cases}
y_t - \Delta y + F^\prime(u^*)y - \frac{\psi}{\mu} = g, &\text{ in } Q,\\
-\psi_t - \Delta \psi + F^\prime(u^*)\psi = h, &\text{ in } Q,\\
y = 0 \text{ and } \psi = 0, &\text{ on } \Sigma,\\
y(0) = y_0 \text{ and } \psi(T) = y(T) + \psi_T, &\text{ in } \Om.
\end{cases}
$$
But this can be easily established by means of standard arguments.
\end{proof}

\begin{algorithm}
\SetAlgoLined
\KwResult{The Pareto optimal control $v,$ and the corresponding state $u.$}
Initialize with the error variable $\epsilon,$ the tolerance $\epsilon_0,$ $k=0,$ an initial guess for the controlled state $u^0,$ and let $u^{\left(\alpha\right)}$ the uncontrolled trajectory corresponding to the initial data $\alpha u_{01} + (1-\alpha)u_{02}.$\\
\While{$\epsilon > \epsilon_0$}{
$1.$ Let $(y^k,\psi^k)$ solve
$$
\Phi^\prime(u^{(\alpha)},0)\cdot (y^k,\psi^k) = \Phi\left(u^k,\varphi^k\right) -  \left( 0,0, u_0, -u_{(\alpha)}^T\right),
$$
where $u_{(\alpha)}^T := \alpha u_1\left(T,\,\cdot\right) + \left(1-\alpha\right) u_2\left(T,\,\cdot\right);$\\
$2.$ Update $\left(u^{k+1},\varphi^{k+1} \right) \gets \left( u^k,\varphi^k\right) - (y^k,\psi^k);$\\
$3.$ Set $v := -\varphi^{k+1}/\mu;$\\
$4.$ Update $\epsilon;$\\
$5.$ $k\gets k+1;$
}
\caption{Newton-Raphson iterative algorithm for the semilinear model}
\end{algorithm}

\begin{theorem}
Let us assume $u_0,\,u_{01},\,u_{02} \in H^1_0(\Om).$ There exists $\delta> 0$ such that, if the initial data satisfy $\|u_0 - u_{01}\|_{H^1_0(\Om)} + \|u_0 - u_{02}\|_{H^1_0(\Om)} \leqslant \delta,$ then Algorithm 4 converges, i.e., $u^k \rightarrow u,$ strongly in $Y,$ and $v^k \rightarrow v,$ strongly in $L^2\left( \om \times \left(0,T\right) \right).$
\end{theorem}
\begin{proof}
Let us denote by $u^{(\alpha)}$ the uncontrolled state corresponding to the initial data $\alpha u_{01} + \left( 1- \alpha \right)u_{02}.$ We introduce $C_0 := \left\|\Phi^\prime\left( u^{(\alpha)},0\right)^{-1}\right\|_{\mathcal{L}(Y;Z)},$ and we take $\delta > 0$ such that
$$
\left\| \Phi^\prime\left( u,\varphi \right) - \Phi^\prime\left( u^{\left(\alpha\right)},0\right) \right\|_{\mathcal{L}(Y;Z)} \leqslant \frac{1}{2C_0}\ \forall (u,\varphi)\in Y\ \text{such that}\ \|(u,\varphi)-(u^{(\alpha)},0)\|_Y \leqslant \delta.
$$

Then we observe that
\begin{align*}
    \left( u^{k+1}, \varphi^{k+1} \right) =& -\Phi^\prime\left( u^{(\alpha)},0\right)^{-1}\left[ \Phi\left(u^k,\varphi^k\right) - \Phi\left( u^{(\alpha)},0 \right) - \Phi^\prime\left(u^{\left(\alpha\right)},0\right)\left(u^k,\varphi^k\right) \right] \\
    &- \Phi^\prime\left( u^{(\alpha)}, 0\right)^{-1}\left[ \Phi\left( u^{(\alpha)},0\right) - \left( 0,0,u_0,-u^{\left(\alpha\right)}\left(\cdot,T\right) \right) \right] \\
    =& -\Phi^\prime\left( u^{(\alpha)},0\right)^{-1}\left[ \Phi\left(u^k,\varphi^k\right) - \Phi\left( u^{(\alpha)},0 \right) - \Phi^\prime\left(u^{\left(\alpha\right)},0\right)\left(u^k,\varphi^k\right) \right] \\
    &- \Phi^\prime\left( u^{(\alpha)}, 0\right)^{-1} \left(  0, 0, \alpha u_{01} + (1-\alpha)u_{02} - u_0, u^{\left(\alpha\right)}\left(\cdot,T\right) -\alpha u_1(\cdot, T) - (1-\alpha) u_2(\cdot,T)\right),
\end{align*}
whence
\begin{align*}
    ( u^{k+1}, \varphi^{k+1} )& - ( u^{( \alpha )},0 ) = -\Phi^\prime( u^{(\alpha)},0)^{-1}\{ \Phi(u^k,\varphi^k) - \Phi( u^{(\alpha)},0 ) - \Phi^\prime(u^{(\alpha)},0)[ (u^k,\varphi^k) - ( u^{(\alpha)}, 0) ] \} \\
    &- \Phi^\prime( u^{(\alpha)}, 0)^{-1} (  0, 0, \alpha u_{01} + (1-\alpha)u_{02} - u_0, u^{(\alpha)}(\cdot,T) -\alpha u_1(\cdot, T) - (1-\alpha) u_2(\cdot,T)).
\end{align*}
Therefore,
\begin{align} \label{eq:InductiveStep_NewtRaph_Semilinear}
  \begin{split}
    \|( u^{k+1}, \varphi^{k+1} ) - ( u^{( \alpha )},0 )\|_Y \leqslant&\, C_0 \sup_{(u,\varphi) \in [(u^k,\varphi^k),(u^{(\alpha)},0)]} \|\Phi^\prime(u,\varphi) - \Phi^\prime(u^{(\alpha)},0)\|_{\mathcal{L}(Y;Z)} \|( u^{k}, \varphi^{k} ) - ( u^{( \alpha )},0 )\|_Y \\
    &+ C\|u_0 - u_{01}\| + C\|u_0 - u_{02}\|,
  \end{split}
\end{align}
where $C=C\left( \Om,\,T,\, F,\, u_{01},\, u_{02}\right).$ If $\|(u^k,\varphi^k) -(u^{(\alpha)},0)\| \leqslant \delta,$ the inequality \eqref{eq:InductiveStep_NewtRaph_Semilinear} implies
$$
\left\|\left( u^{k+1}, \varphi^{k+1} \right) - \left( u^{\left( \alpha \right)},0 \right)\right\|_Y \leqslant \frac{\delta}{2} + C\|u_0 - u_{01}\|_V + C\|u_0 - u_{02}\|_V \leqslant \delta,
$$
as long as $C\|u_0 - u_{01}\| + C\|u_0 - u_{02}\| \leqslant \delta/2.$ This way, we deduce by induction that the sequence $\left\{ (u^k,\varphi^k) \right\}$ is in fact confined to the $W-$ball centered at $\left(u^{\left(\alpha\right)},0\right)$ with radius $\delta.$ We can use this to estimate the distance between iterates:
\begin{align*}
    \| ( u^{k+1}, \varphi^{k+1} ) - ( u^k,\varphi^k) \|_Y &\leqslant C_0 \sup_{(u,\varphi) \in [( u^k,\varphi^k),( u^{k-1},\varphi^{k-1})] } \|\Phi^\prime(u,\varphi) - \Phi^\prime( u^{(\alpha)},0) \|_{\mathcal{L}(Y;Z)} \| ( u^k,\varphi^k)- ( u^{k-1},\varphi^{k-1})\|_Y \\
    &\leqslant \frac{1}{2}\| ( u^k,\varphi^k)- ( u^{k-1},\varphi^{k-1})\|_Y.
\end{align*}
We conclude that the iteration process is convergent, thus finishing the proof.
\end{proof}

\subsection{Bilinear control problem} \label{subsec:Bilinear_Algos}
Here, we consider Pareto equilibria for (\ref{pareto-28}) and (\ref{def:ObjFuncs}). The optimality systems for Pareto equilibria are given by (\ref{pareto-28}), (\ref{eq:Bilinear_AdjointPDEs}) and (\ref{paretoeq_linear_mupos}) where (again) $\alpha \in [0,1]$.
\subsubsection{A gradient descent algorithm}

\begin{algorithm}
\SetAlgoLined
\KwResult{The optimal control $v.$}
Initialize with the error variable $\epsilon,$ the tolerance $\epsilon_0,$ $k=0,$ the descent speed $\tau,$ and an initial guess $v^0$ for the control $v.$\\
\While{$\epsilon > \epsilon_0$}{
$1.$ Compute the solution $u^{k+1}$ of \eqref{pareto-28} with $v=v^k;$\\
$2.$ Calculate the solution $\varphi^{k+1}$ of \eqref{eq:Bilinear_AdjointPDEs} with $u=u^{k+1}$ and $v=v^k;$\\
$3.$ Update $v:=v^{k+1} \gets P_\mathcal{U}\left( v^k -\tau J^\prime\left(v^{k}\right) \right) = P_\mathcal{U}\left( \left( 1- \tau \mu \right)v^k - \tau \varphi^{k+1} u^{k+1} \right);$\\
$4.$ Update $\epsilon;$\\
$5.$ $k\gets k+1.$
}
\caption{Gradient descent algorithm for the bilinear problem}
\end{algorithm}

\begin{lemma} \label{lem:LipGradBilinear}
Let us suppose that $N\leqslant 3,$ $u_0 \in H^1_0(\Om)$ and, for some $R>0,$ $\mathcal{U} \subseteq \left\{ v : |v|\leqslant R\right\}.$ Then the mapping $v \mapsto \varphi u$ is Lipschitz-continuous.
\end{lemma}
\begin{proof}
Let us denote by $u^v$ and $\varphi^v$ the state and adjoint state corresponding to the control $v$. It is straightforward to prove that
$$
\sup_{\left[0,T\right]}\left(\|\nabla u^v\|^2 + \|\Delta u^v\|^2\right)+ \sup_{\left[0,T\right]}\left(\|\nabla \varphi^v\|^2 + \|\Delta \varphi^v\|^2\right)\leqslant C\| u_0 \|_{H^1_0(\Om)}^2.
$$

For any $v,v^\prime \in \mathcal{U},$ let us set $\delta u := u^v - u^{v^\prime}$ and $\delta \varphi := \varphi^v - \varphi^{v^\prime}.$ Then,
\begin{align*}
    \frac{1}{2}\frac{d}{dt}\left( \int_\Om |\delta u|^2\,dx\right) + \int_\Om |\nabla \delta u|^2\,dx &= \int_\Om v (\delta u)^2\,dx + \int_\Om u^{v^\prime}(v-v^\prime)\delta u \,dx \\
    &\leqslant C \left( \|\delta u \|^2 + \|\nabla u^{v^\prime}\|\|v-v^\prime\|\|\nabla \delta u\| \right) \\
    &\leqslant C \left( \|\delta u \|^2 + \|\delta v\|^2 \right) + \frac{1}{2}\|\nabla \delta u\|^2,
\end{align*}
whence
$$
\sup_{\left[0,T\right]}\left(\|\delta u\|^2 + \|\nabla \delta u\|^2\right) \leqslant C\|v-v^\prime\|^2.
$$
Similar estimates give
$$
\sup_{\left[0,T\right]}\|\delta \varphi\|^2 + \|\nabla \delta \varphi\|^2 \leqslant C\|v-v^\prime\|^2.
$$
Finally, we deduce that
\begin{align*}
    \iint_Q \left| u^v \varphi^v - u^{v^\prime}\varphi^{v^\prime}\right|^2\,dx\,dt &\leqslant C\iint_Q \left( \left|\delta u\right|^2 \left|\varphi^v \right|^2 + \left|u^{v^\prime}\right|^2\left|\delta \varphi \right|^2 \right)\,dx\,dt \\
    &\leqslant \int_0^T \left[ \|\Delta \varphi^{v}\|^2 \|\delta u\|^2 + \|\Delta u^{v^\prime}\|^2 \|\delta \varphi\|^2 \right]\,dt \\
    &\leqslant C \|v-v^\prime\|^2.
\end{align*}
\end{proof}

From Lemma \ref{lem:LipGradBilinear} and Theorem 10.5.8 in \cite{allaire2007numerical}, the following convergence result holds:

\begin{theorem}
Under the assumptions of Lemma \ref{lem:LipGradBilinear}, if $\mu/\|u_0\|$ is sufficiently large and the descent step $\tau > 0$ is sufficiently small, then Algorithm $5$ converges, i.e., $v^k \rightarrow v$ strongly in $L^2\left( \om \times \left(0,T\right)\right).$
\end{theorem}

\subsubsection{An iterative fixed-point algorithm}

\begin{algorithm}
\SetAlgoLined
\KwResult{The optimal control $v.$}
Initialize with the error variable $\epsilon,$ the tolerance $\epsilon_0,$ $k=0,$ and an initial guess $v^0$ for the control $v.$\\
\While{$\epsilon > \epsilon_0$}{
$1.$ Compute the solution $u^{k+1}$ of \eqref{pareto-28} with $v=v^k;$\\
$2.$ Compute the solution $\varphi^{k+1}$ of \eqref{eq:Bilinear_AdjointPDEs} with $u=u^{k+1}$ and $v=v^k;$\\
$3.$ Update $v:=v^{k+1} \gets P_\mathcal{U}\left( - u^{k+1} \varphi^{k+1} / \mu \right);$\\
$4.$ Update $\epsilon;$\\
$5.$ $k\gets k+1.$
}
\caption{Iterative algorithm for the bilinear problem}
\end{algorithm}

\begin{theorem}
Let us assume that $N\leqslant 3,$ $\mathcal{U} \subseteq \left\{v: |v|\leqslant R\right\}$ and also that the projection $P_\mathcal{U}$ has the following property: $$ 
\text{If}\ f^k \rightarrow f \ \text{a.e. in $Q$, then}\  P_\mathcal{U}(f^k) \rightarrow P_\mathcal{U}(f) \text{ a.e. in } Q.
$$ 
Then, the sequence $\left\{ v^k\right\}$ furnished by Algorithm $6$ converges to a Pareto equilibrium $v$, that is, $v^k \rightarrow v$ strongly in $L^1(Q)$ and a.e. in $Q.$
\end{theorem}
\begin{proof}
Carrying out standard energy estimates in the equation satisfied by $u^k,$ we obtain:
$$
\sup_{\left[0,T\right] }\left(\|u^k(t)\| + \|\nabla u^k\|\right) \leqslant C.
$$
Likewise, the same estimates for $\varphi^k$ yield
$$
\sup_{\left[0,T\right]}\left(\|\varphi^k(t)\| + \|\nabla \varphi^k \|\right) \leqslant C\left\|u^k\left(\cdot,T\right)\right\| \leqslant C.
$$
Arguing as usual, we can assume that there exist $u,\varphi \in L^2(Q)$ for which $u^k \rightarrow u$ and $\varphi^k \rightarrow \varphi$ a.e. in $Q,$ as well as strongly in $L^2(Q);$ hence $u^k \varphi^k \rightarrow u \varphi$ a.e. in $Q$ and strongly in $L^1(Q).$ The result now follows easily.
\end{proof}




\eject

\section{Computation of equilibria (II): Numerical experiments}\label{sec:Numerics2}

\subsection{Introduction}

In this section, we intend to illustrate the behavior of the algorithms in Section \ref{sec:Numerics1}. For each model, i.e. the linear, semilinear and bilinear systems, we present 2D and 3D results. More precisely, the data for problem (\ref{def:ObjFuncs})-(\ref{pareto-1}) have been the following:
{ \begin{itemize}

    \item Consider $\mathcal{B}_{r}(x_0,y_0)$ the ball centered in $(x_0,y_0)$ with r-radius;
    
    \item $OX_{\omega}:= [-1.5, \, 1.5]$,  \ $O_1X_{\omega}:=[-1.5, \, 0.3]$, \ $O_2X_{\omega}:=[-0.3, \, 1.5]$, \ $OY_{\omega}:=[0, \, 1.5]$, and $OZ:=[0, \, 3]$;
    
    \item If $N=2$, $\Omega= \mathcal{B}_{3}(0,0)$, \  $\om= OX_{\omega} \times OY_{\omega}$, $T=0.5$, \ $\mathcal{O}_1= O_1X_{\omega} \times OY_{\omega}$ and \ $\mathcal{O}_2= O_2X_{\omega} \times OY_{\omega}$;
    
    For the linear and semilinear cases, we consider the initial dates: 
    
    \begin{equation*}
    \begin{cases}
     u_{0}(x,y)=0;\\
     u_{0,1}(x,y)= 3-\sqrt{(x^2+y^2)};\\ 
     u_{0,2}(x,y)= \sqrt{(x^2+y^2)}-3.
     \end{cases}
    \end{equation*}
    
    in the bilinear case, we consider: 
    
      \begin{equation*}
    \begin{cases}
     u_{0}(x,y)=x^3y^3\sin\biggl(\displaystyle\frac{2\pi}{3}\sqrt{(x^2+y^2)}\biggr);\\
     u_{0,1}(x,y)= 3-\sqrt{(x^2+y^2)};\\ 
     u_{0,2}(x,y)= \sqrt{(x^2+y^2)}-3.
     \end{cases}
    \end{equation*}
 
      \item If $N=3$, $\Omega= \mathcal{B}_{3}(0,0) \times OZ$, \  $\om= OX_{\omega} \times OY_{\omega} \times OZ$, $T=0.5$, \ $\mathcal{O}_1= O_1X_{\omega} \times OY_{\omega} \times OZ$ and \ $\mathcal{O}_2= O_2X_{\omega} \times OY_{\omega} \times OZ$. 
      
     Now, for the linear and semilinear cases, we consider:  
    
    \begin{equation*}
    \begin{cases}
      u_{0}(x,y,z)=0;\\
     u_{0,1}(x,y,z)= \bigl(3-\sqrt{(x^2+y^2+z^2)}\bigr);\\ 
     u_{0,2}(x,y,z)=  \bigl(\sqrt{(x^2+y^2+z^2)}-3\bigr).
      \end{cases}
    \end{equation*}
    
    and finally, the bilinear case we consider: 
  
    \begin{equation*}
    \begin{cases}
     \displaystyle u_{0}(x,y,z)=x^3 y^3z^2\sin\biggl(\frac{2\pi}{3}z\sqrt{(x^2+y^2)}\biggr);\\
     u_{0,1}(x,y,z)= \bigl(3-\sqrt{(x^2+y^2)}\bigr)z;\\ 
     u_{0,2}(x,y,z)= \bigl(\sqrt{(x^2+y^2)}-3\bigr)z.
     \end{cases}
    \end{equation*}
   
\end{itemize}}

The meshes used in the numerical tests are displayed in Fig. 1. On the other hand, the targets, that is, the $u_i$ at time $T$, are depicted in Fig. 2 and 3.

\begin{figure}[H]
\centering
\includegraphics[scale = 0.4]{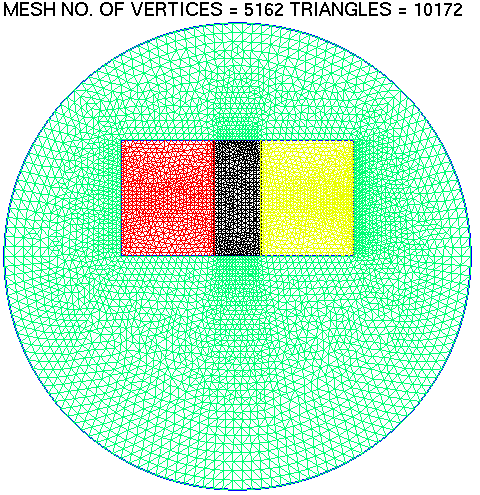}
\includegraphics[scale = 0.4]{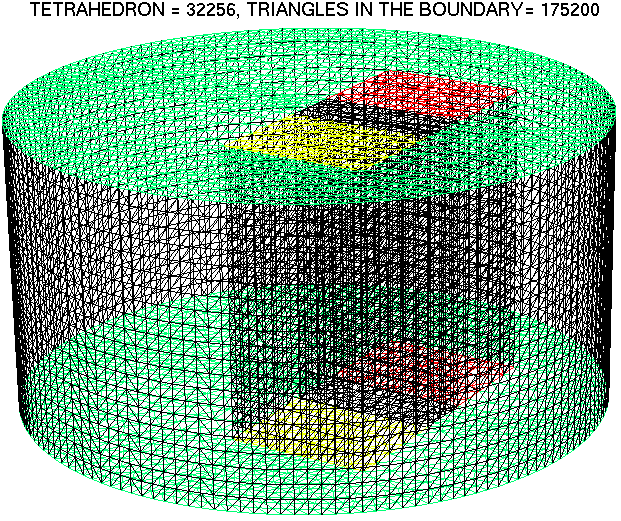}
\caption{The mesh of the domains. In the left, we showcase that of the 2D domain, and in the right, the one we used in the 3D experiments.}
\label{fig:domains_meshes}
\end{figure}

\begin{figure}[H]
\centering
\includegraphics[scale = 0.42]{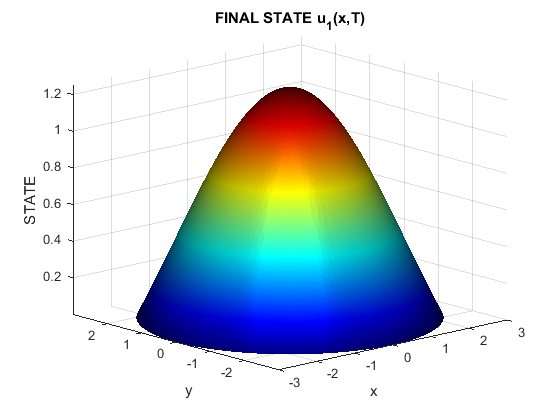} 
\includegraphics[scale = 0.42]{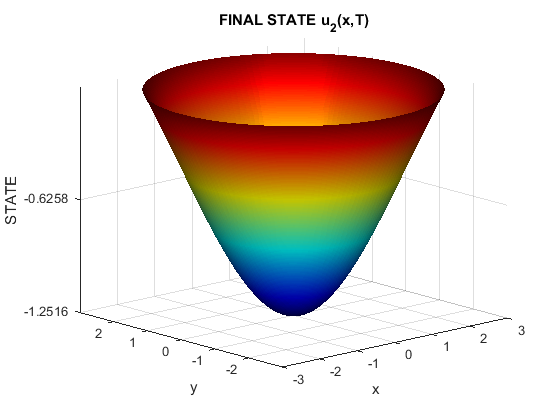}
\caption{Surfaces of the trajectories at terminal time, $u_1(\cdot,T)$ (left) and $u_2(\cdot,T)$ (right), that we used in the experiments of the 2D problems.}
\label{fig:targets_2D}
\end{figure}

\begin{figure}[H]
\centering
\includegraphics[scale = 0.42]{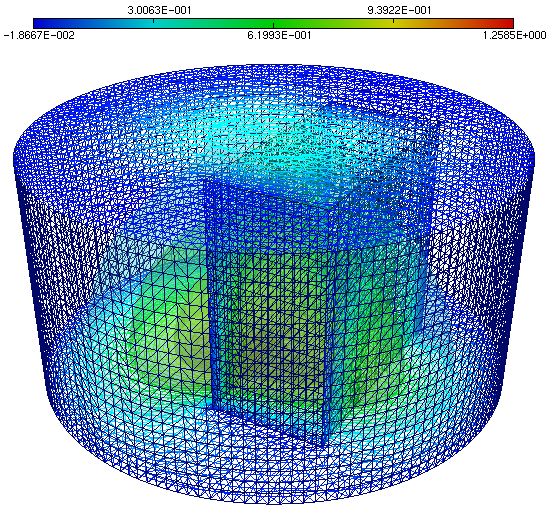}
\includegraphics[scale = 0.42]{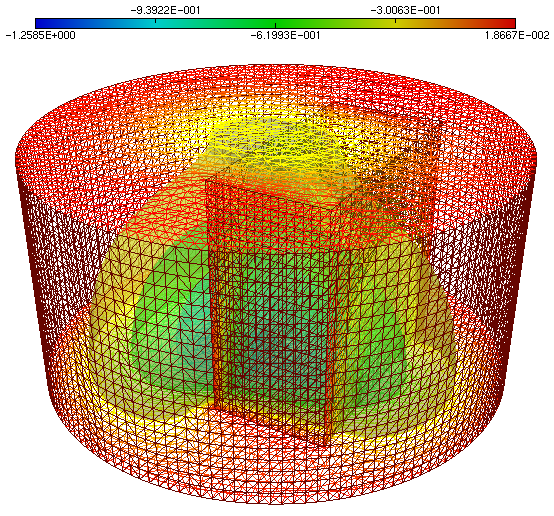}
\caption{The values of the trajectories at terminal time, $u_1(\cdot,T)$ (left) and $u_2(\cdot,T)$ (right), that we used in the experiments of the 3D problems.}
\label{fig:targets_3D}
\end{figure}

\subsection{Linear model}

The problem to solve in this case is (\ref{pareto-1}), (\ref{pareto-5}), (\ref{paretoeq_linear_mupos}). This has been achieved with Algorithm $1$. The final states corresponding to several values of $\alpha$ have been shown in Fig. 4.

\subsubsection{Experiments in the 2D case (Test 1)}\label{COMP-L-2D}



\begin{figure}[H]
\centering
\includegraphics[scale = .33]{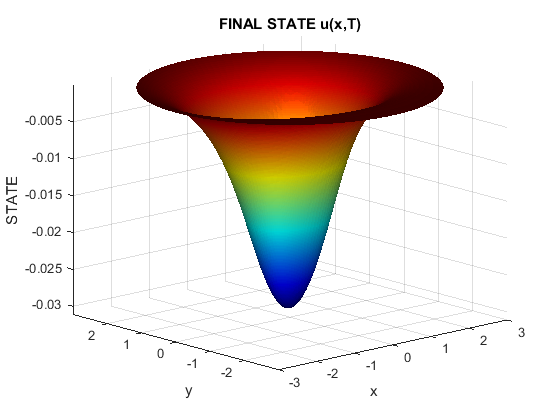}
\includegraphics[scale = .33]{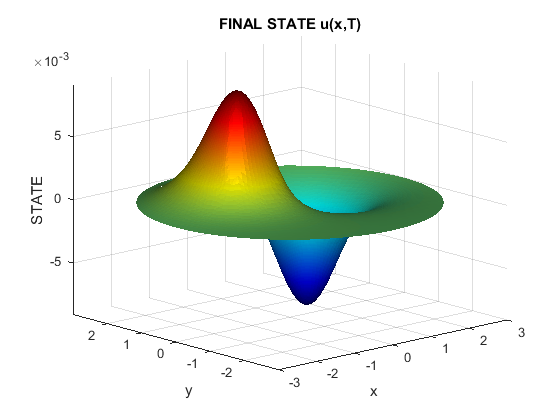}
\includegraphics[scale = .33]{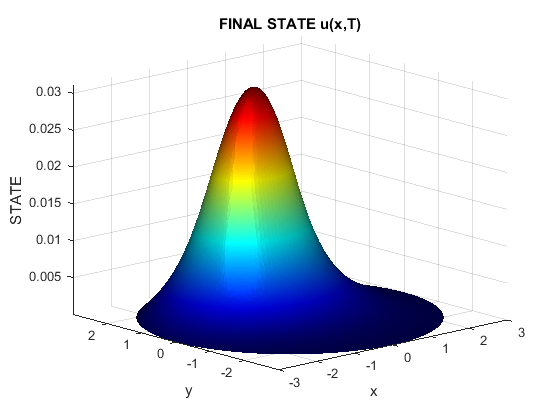}
\caption{Test 1 - The final state $u(\cdot,T),$ under the action of the Pareto optimal control $v$ corresponding to $\mu=5$ and $\alpha=0.05$ (left), $\alpha=0.5$ (middle), and $\alpha=0.95$ (right).}
\end{figure}

We have also shown the norms $\left\|u(T)-u_1(T)\right\|$ and $\left\|u(T)-u_2(T)\right\|$ associated to several values of $\mu$ and $\alpha$ in Fig. 5.

\begin{figure}[H]
\centering
 \includegraphics[scale = 0.4]{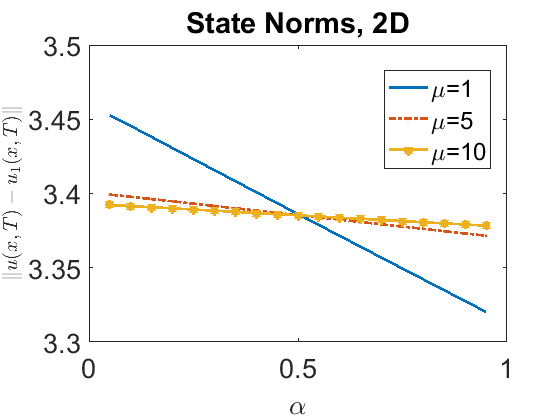}
\includegraphics[scale = 0.4]{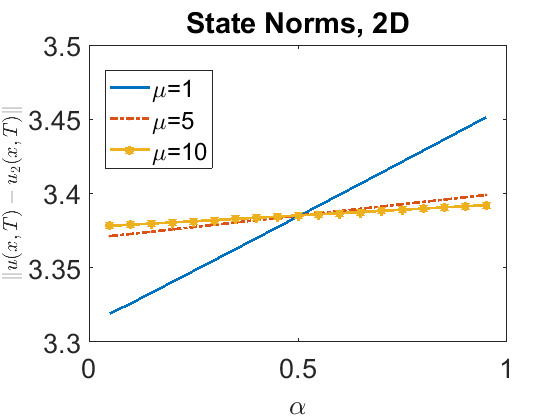}
\caption{Test 1 - The norms $\left\|u(T)-u_1(T)\right\|$ and $\left\|u(T)-u_2(T)\right\|$ corresponding to the computed Pareto equilibria for $\mu \in \left\{ 1,\, 5,\,10\right\}$ and various $\alpha \in \left]0, 1\right[.$}
\end{figure}

On the other hand, the norm of the computed controls have been depicted in Fig. 6. Also, the numbers of iterates appears in Fig. 7.

\begin{figure}[H]
    \centering
    \includegraphics[scale = .5]{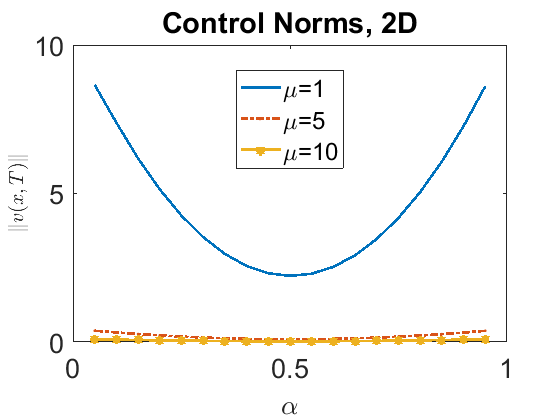}
    \caption{Test 1 - The norms $\|v\|_{L^2\left( \left(0,T\right)\times \omega\right)}$ of the computed Pareto equilibria for $\mu \in \left\{ 1,\, 5,\, 10\right\}$ and various $\alpha \in \left]0,1\right[.$}
\end{figure}

\begin{figure}[H]
\centering
\includegraphics[scale = 0.5]{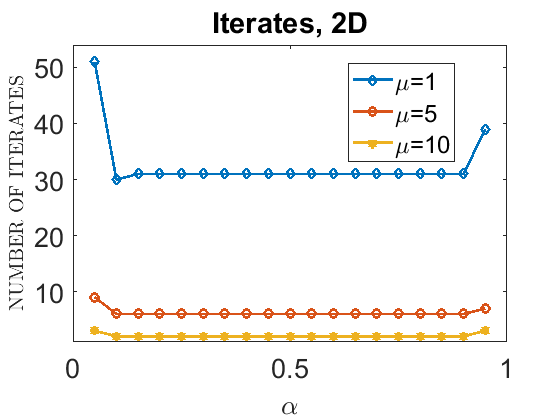}
\caption{Test 1 - The numbers of iterates needed to attain convergence with Algorithm $1$. Tolerance = $10^{-8}$.}
\end{figure}

\subsubsection{Experiments in the 3D case (Test 2)}\label{COMP-L-3D}
We have shown the norms $\left\|u(T)-u_1(T)\right\|$ and $\left\|u(T)-u_2(T)\right\|$ associated to several values of $\mu$ and $\alpha$ in Fig. 8. 
\begin{figure}[H]
\centering
\includegraphics[scale = 0.4]{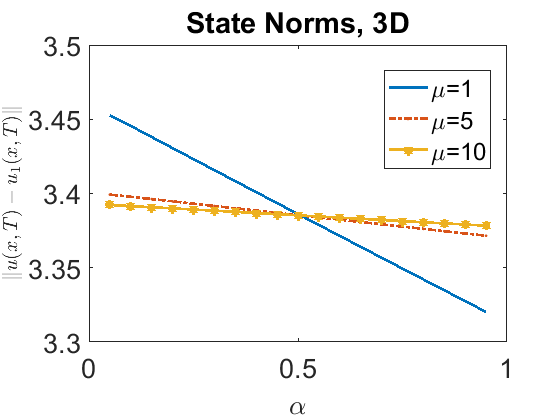}
\includegraphics[scale = 0.4]{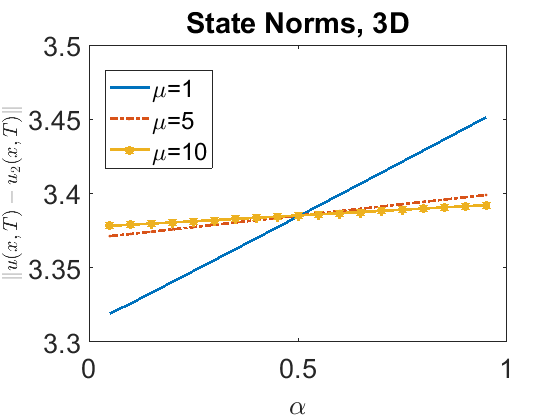}
\caption{Test 2 - The norms $\left\|u(T)-u_1(T)\right\|$ and $\left\|u(T)-u_2(T)\right\|$ corresponding to the computed Pareto equilibria for $\mu \in \left\{ 1,\, 5,\,10\right\}$ and various $\alpha \in \left]0, 1\right[.$}
\end{figure}
On the other hand, the norms of the computed controls have been depicted in Fig. 9. Also, the number of iterates appears in Fig. 10.
\begin{figure}[H]
\centering
\includegraphics[scale = 0.5]{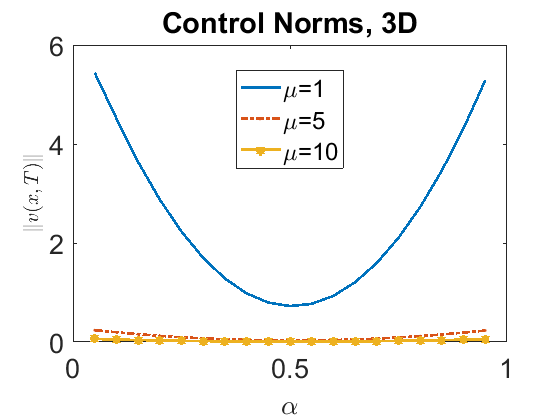}
\caption{Test 2 - The norms $\|v\|_{L^2\left( \left(0,T\right) \times \omega \right)}$ of the computed Pareto equilibria for $\mu \in \left\{ 1,\, 5,\, 10 \right\}$ and various $\alpha \in \left]0,1\right[.$}
\end{figure}

\begin{figure}[H]
\centering
\includegraphics[scale = 0.5]{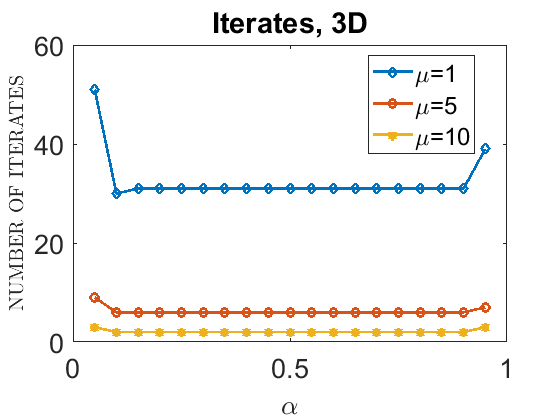}
\caption{Test 2 - The number of iterates needed to attain convergence with Algorithm 1. Tolerance = $10^{-8}$.}
\end{figure}

\subsection{Semilinear model}

Here, we try to compute Pareto equilibria for (\ref{pareto-15}) and (\ref{def:ObjFuncs}). The function $F$ is given by $${ F(s)= s\cdot(1+sin(s)), \ \ \ \forall s\in \mathbb{R}. }$$

\subsubsection{Experiments for the 2D case (Test 3)}\label{COMP-SL-2D}
\begin{center}
    \textbf{Algorithm 3}
\end{center}




\begin{figure}[H]
\centering
\includegraphics[scale = 0.4]{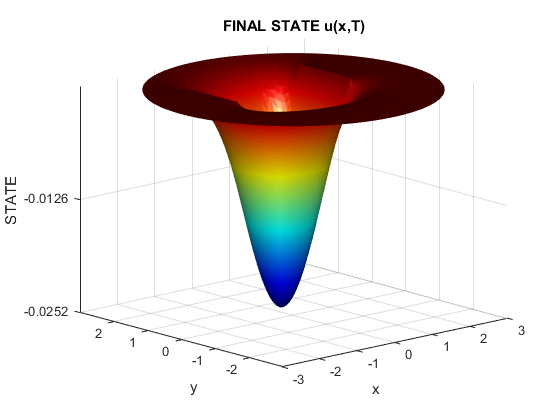}
\includegraphics[scale = 0.4]{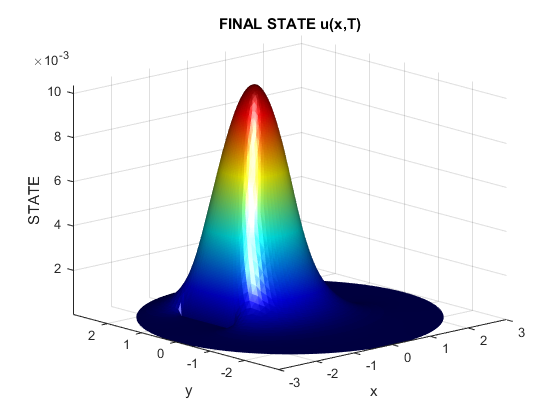}
\caption{Test 3 - The final state $u(\cdot,T)$ under the action of the Pareto optimal control $v$ corresponding to $\mu=5$ and $\alpha=0.05$ (left) and $\alpha = 0.95$ (right).}
\end{figure}

We have also shown the norms $\left\|u(T)-u_1(T)\right\|$ and $\left\|u(T)-u_2(T)\right\|$ associated to several values of $\mu$ and $\alpha$ in Fig. 12.

\begin{figure}[H]
\centering
\includegraphics[scale = 0.4]{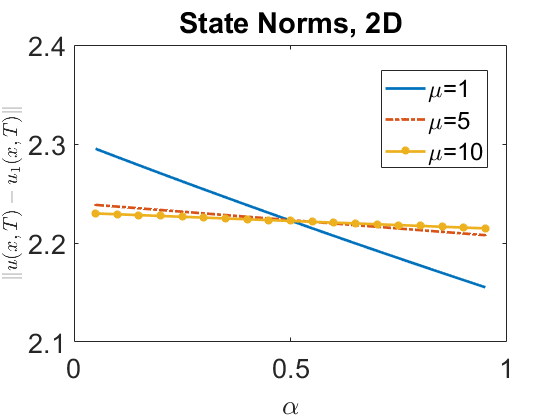}
\includegraphics[scale = 0.4]{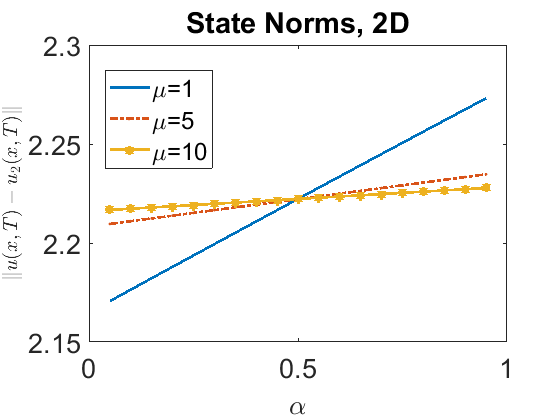}
\caption{Test 3 - The norms of $\|u(T)-u_1(T)\|$ and $\|u(T)-u_2(T)\|$ corresponding to the computed Pareto equilibria for $\mu\in \left\{ 1,\, 5,\,10\right\}$ and various $\alpha \in \left] 0, 1 \right[.$}
\end{figure}

On the other hand, the norms of the computed controls have been depicted in Fig. 13. Also, the number of iterates appears in Fig. 14.

\begin{figure}[H]
\centering
\includegraphics[scale = 0.5]{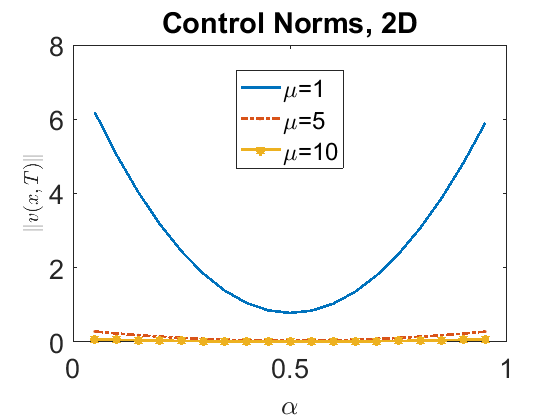}
\caption{Test 3 - The norms $\|v\|_{L^2\left( \left(0,T\right) \times \omega \right)}$ of the computed Pareto equilibria for $\mu \in \left\{ 1,\, 5,\,10\right\}$ and various $\alpha\in\left] 0,1\right[.$}
\end{figure}

\begin{figure}[H]
\centering
\includegraphics[scale = 0.5]{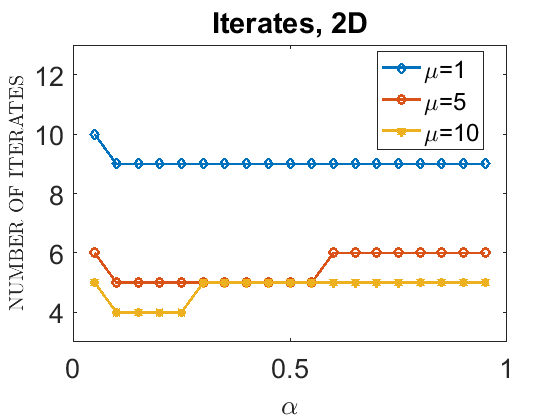}
\caption{Test 3 - The numbers of iterates needed for convergence. Tolerance = $10^{-8}$.}
\end{figure}

\begin{center}
    \textbf{Algorithm 4}
\end{center}




\begin{figure}[H]
\centering
\includegraphics[scale = 0.4]{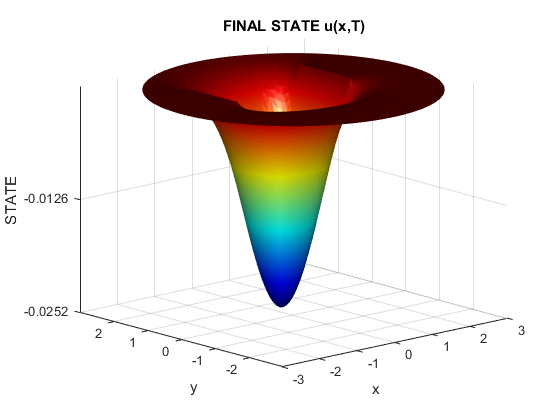}
\includegraphics[scale = 0.4]{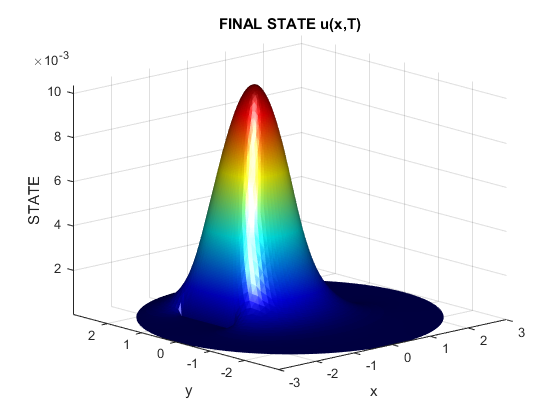}
\caption{Test 3 - The final state $u(\cdot,T)$ under the action of the Pareto optimal control $v$ corresponding to $\mu=5$ and $\alpha=0.05$ (left) and $\alpha = 0.95$ (right).}
\end{figure}

We have also shown the norms $\left\|u(T)-u_1(T)\right\|$ and $\left\|u(T)-u_2(T)\right\|$ associated to several values of $\mu$ and $\alpha$ in Fig. 16.

\begin{figure}[H]
\centering
\includegraphics[scale = 0.4]{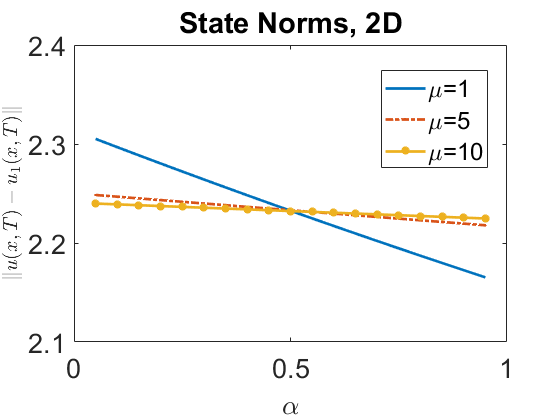}
\includegraphics[scale = 0.4]{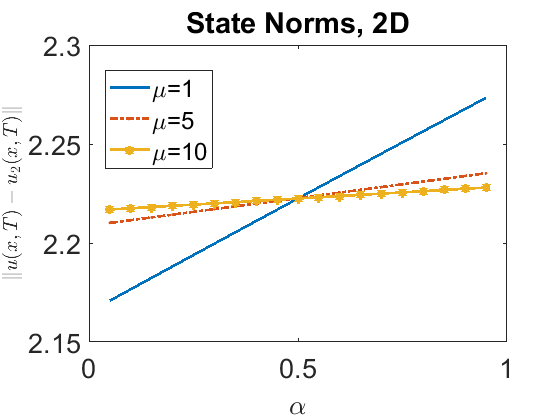}
\caption{Test 3 - The norms of $\|u(\cdot,T)-u_1(\cdot,T)\|$ and $\|u(\cdot,T)-u_2(\cdot,T)\|$ corresponding to the computed Pareto equilibria for $\mu\in \left\{ 1,\, 5,\,10\right\}$ and various $\alpha \in \left] 0, 1 \right[.$}
\end{figure}

On the other hand, the norms of the computed controls have been depicted in Fig. 17. Also, the number of iterates appears in Fig. 18.

\begin{figure}[H]
\centering
\includegraphics[scale = 0.5]{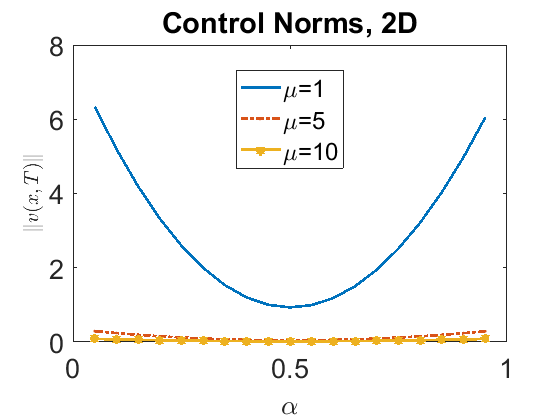}
\caption{Test 3 - The norms $\|v\|_{L^2\left( \left(0,T\right) \times \omega \right)}$ of the computed Pareto equilibria for $\mu \in \left\{ 1,\, 5,\,10\right\}$ and various $\alpha\in\left] 0,1\right[.$}
\end{figure}

\begin{figure}[H]
\centering
\includegraphics[scale = 0.5]{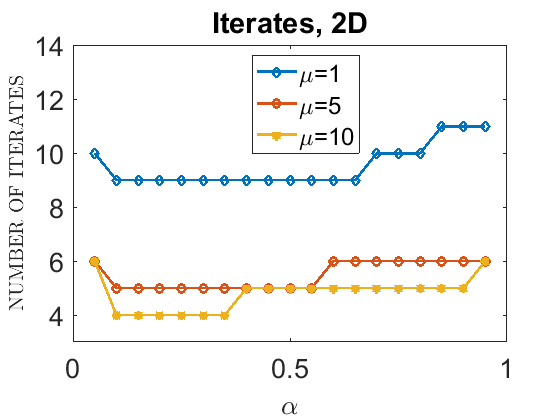}
\caption{Test 3 - The numbers of iterates needed for convergence. Tolerance = $10^{-8}$.}
\end{figure}

\subsubsection{Experiments in the 3D case (Test 4)}

\begin{center}
    \textbf{Algorithm 3}
\end{center}

\begin{figure}[H]
\centering
\includegraphics[scale = 0.4]{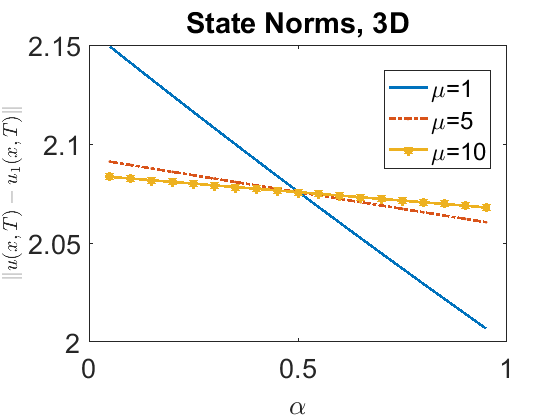}
\includegraphics[scale = 0.4]{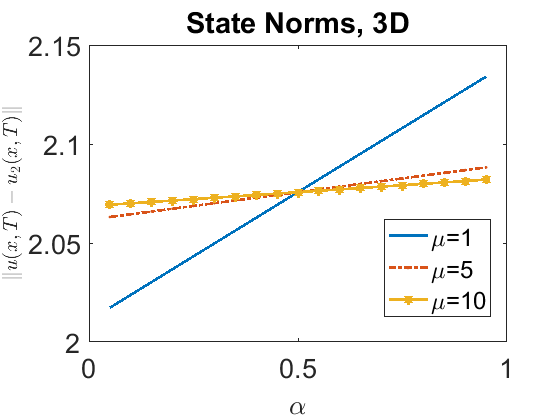}
\caption{Test 4 - Plot of the norms $\|u(T)-u_1(T)\|$ (left) and $\|u(T)-u_2(T)\|$ (right) for the Algorithm 3, with parameters ranging $\mu \in \left\{ 1,\, 5,\, 10 \right\}$ and $\alpha \in \left] 0,1 \right[.$}
\end{figure}

\begin{figure}[H]
\centering
\includegraphics[scale = 0.5]{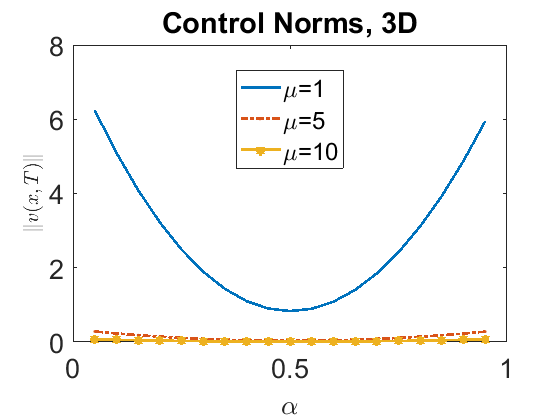}
\caption{Test 4 - Plot of the norms $\|v\|_{L^2\left( \left(0,T\right) \times \omega\right)}$ for parameter values ranging $\mu  \in \left\{ 1,\, 5,\,  10\right\}$ and $\alpha \in \left] 0, 1\right[.$}
\end{figure}

\begin{figure}[H]
\centering
\includegraphics[scale = 0.5]{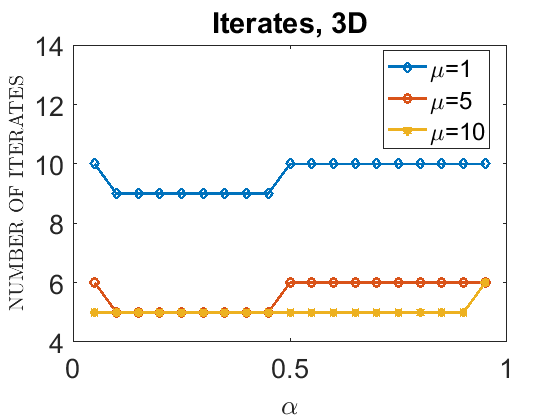}
\caption{Test 4 - Plot of the number of iterates we needed for convergence. Tolerance = $10^{-8}$.}
\end{figure}

\begin{center}
    \textbf{Algorithm 4}
\end{center}

\begin{figure}[H]
\centering
\includegraphics[scale = 0.4]{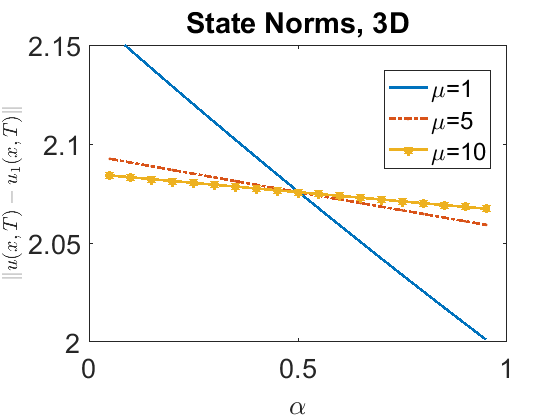}
\includegraphics[scale = 0.4]{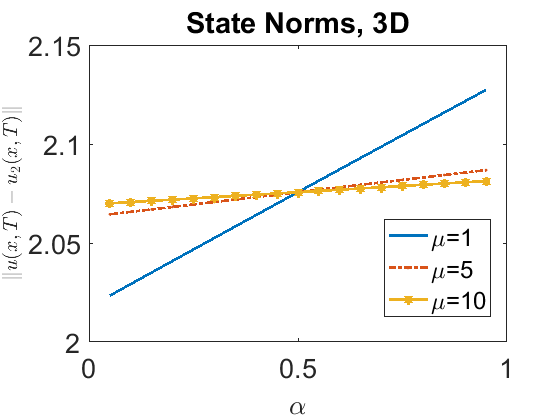}
\caption{Test 4 - Plot of the norms $\|u(T)-u_1(T)\|$ (left) and $\|u(T)-u_2(T)\|$ (right) for the Algorithm 3, with parameters ranging $\mu \in \left\{ 1,\, 5,\, 10 \right\}$ and $\alpha \in \left] 0,1 \right[.$}
\end{figure}

\begin{figure}[H]
\centering
\includegraphics[scale = 0.5]{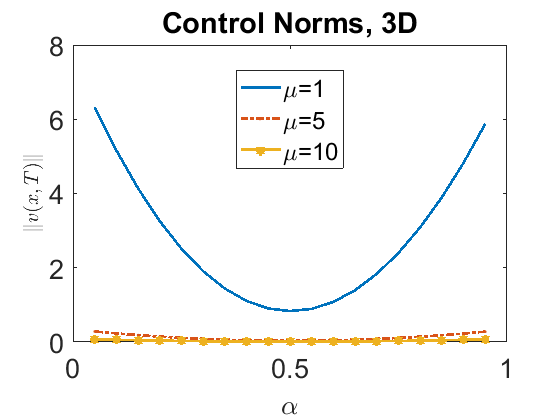}
\caption{Test 4 - Plot of the norms $\|v\|_{L^2\left( \left(0,T\right) \times \omega\right)}$ for parameter values ranging $\mu  \in \left\{ 1,\, 5,\,  10\right\}$ and $\alpha \in \left] 0, 1\right[.$}
\end{figure}

\begin{figure}[H]
\centering
\includegraphics[scale = 0.5]{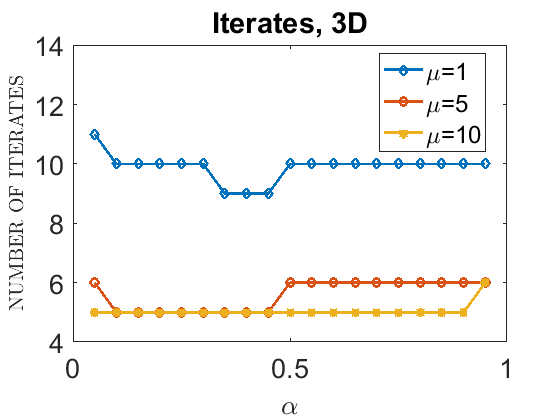}
\caption{Test 4 - Plot of the number of iterates we needed for convergence. Tolerance = $10^{-8}$.}
\end{figure}

\subsection{The bilinear control case} 

\subsubsection{Experiments in  the 2D case (Test 5)}\label{COMP-BL-2D}

\begin{center}
    \textbf{Algorithm 5}
\end{center}

\begin{figure}[H]
\centering
\includegraphics[scale = 0.4]{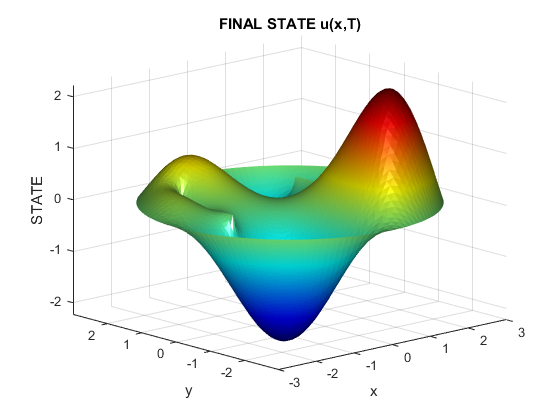}
\caption{Test 5 - Surface of the terminal controlled state $u(\cdot,T)$ corresponding to the parameter $\mu=5$ and $\alpha=0.5.$}
\end{figure}

\begin{figure}[H]
\centering
\includegraphics[scale = 0.4]{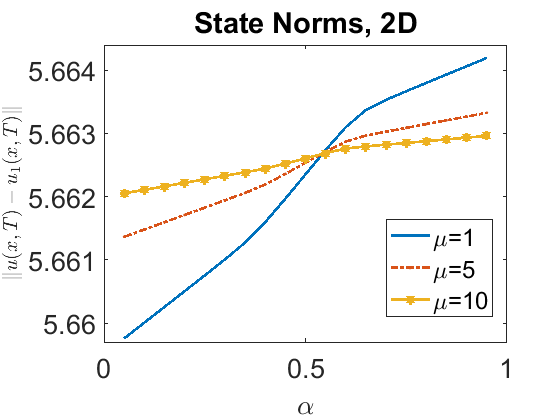}
\includegraphics[scale = 0.4]{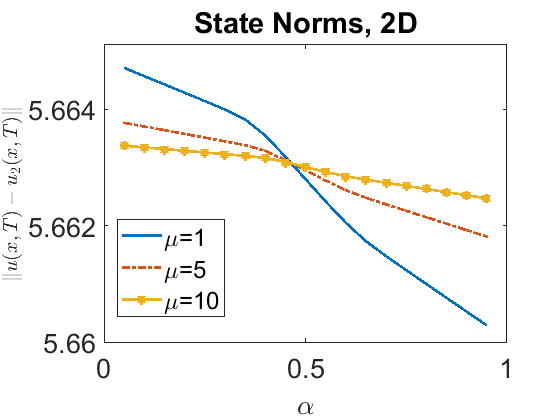}
\caption{Test 5 - Plot of the norms $\|u(T)-u_1(T)\|$ (left) and $\|u(T)-u_2(T)\|$ (right) for the Algorithm 5, with parameters ranging $\mu \in \left\{ 1,\, 5,\, 10 \right\}$ and $\alpha \in \left] 0,1 \right[.$}
\end{figure}

\begin{figure}[H]
\centering
\includegraphics[scale = 0.5]{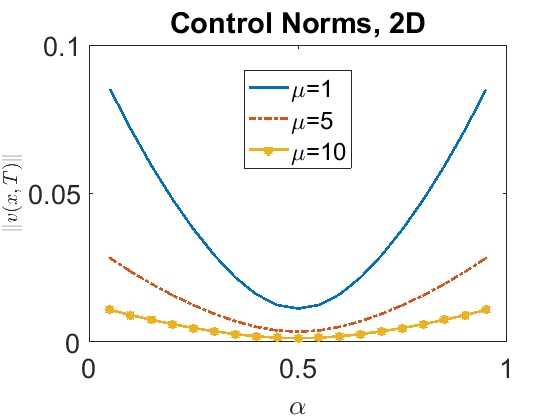}
\caption{Test 5 - Plot of the norms $\|v\|_{L^2\left( \left(0,T\right) \times \omega\right)}$ for parameter values ranging $\mu  \in \left\{ 1,\, 5,\,  10\right\}$ and $\alpha \in \left] 0, 1\right[.$}
\end{figure}

\begin{figure}[H]
\centering
\includegraphics[scale = 0.5]{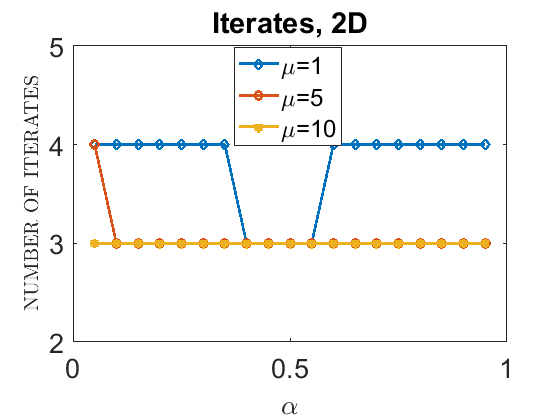}
\caption{Test 5 - Plot of the number of iterates we needed for convergence. Tolerance = $10^{-8}$.}
\end{figure}

\begin{center}
    \textbf{Algorithm 6}
\end{center}

\begin{figure}[H]
\centering
\includegraphics[scale = 0.4]{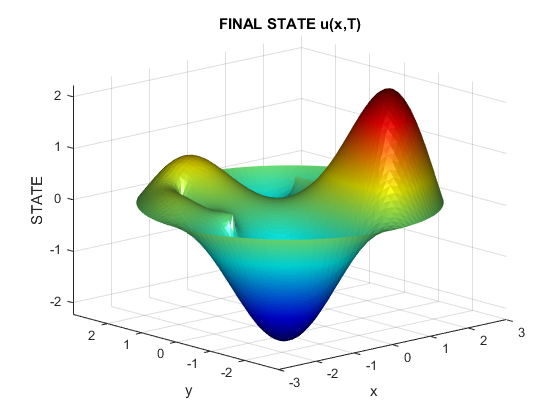}
\caption{Test 5 - Surface of the terminal controlled state $u(\cdot,T)$ corresponding to the parameter $\mu=5$ and $\alpha=0.5.$}
\end{figure}

\begin{figure}[H]
\centering
\includegraphics[scale = 0.4]{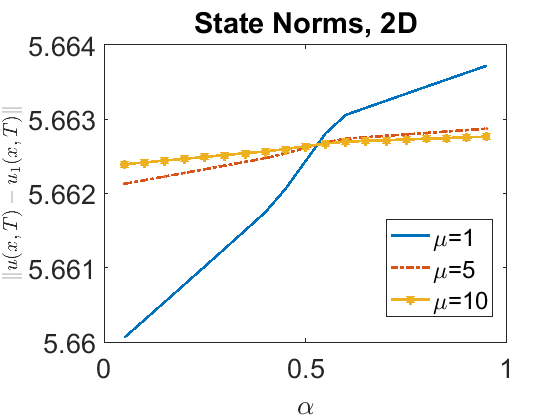}
\includegraphics[scale = 0.4]{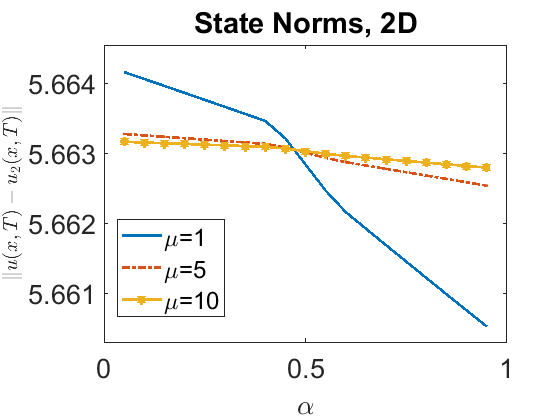}
\caption{Test 5 - Plot of the norms $\|u(\cdot,T)-u_1(\cdot,T)\|$ (left) and $\|u(x,T)-u_2(x,T)\|$ (right) for the Algorithm 6, with parameters ranging $\mu \in \left\{ 1,\, 5,\, 10 \right\}$ and $\alpha \in \left] 0,1 \right[.$}
\end{figure}

\begin{figure}[H]
\centering
\includegraphics[scale = 0.5]{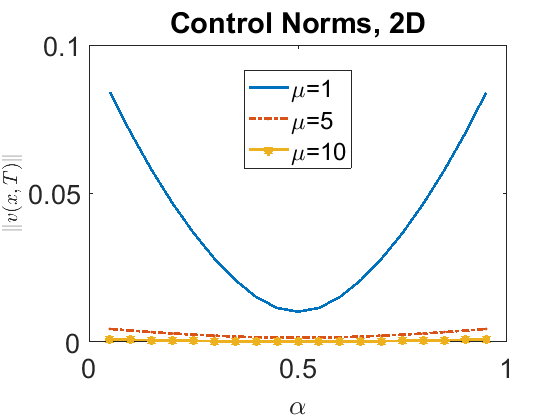}
\caption{Test 5 - Plot of the norms $\|v\|_{L^2\left( \left(0,T\right) \times \omega\right)}$ for parameter values ranging $\mu  \in \left\{ 1,\, 5,\,  10\right\}$ and $\alpha \in \left] 0, 1\right[.$}
\end{figure}

\begin{figure}[H]
\centering
\includegraphics[scale = 0.5]{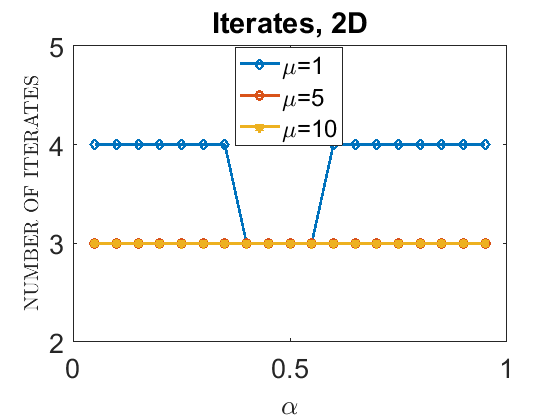}
\caption{Test 5 - Plot of the number of iterates we needed for convergence. Tolerance = $10^{-8}$.}
\end{figure}

\subsection{Experiments in the 3D case (Test 6)}

\begin{center}
    \textbf{Algorithm 5}
\end{center}

\begin{figure}[H]
\centering
\includegraphics[scale = 0.4]{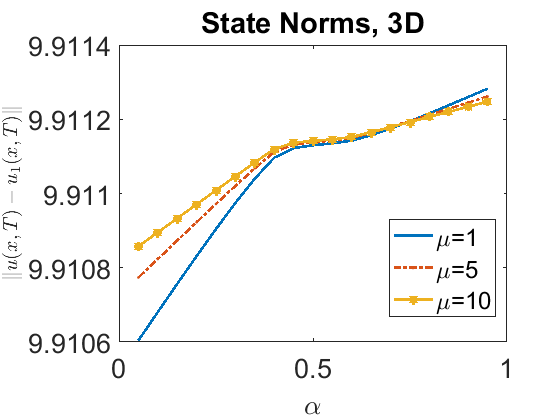}
\includegraphics[scale = 0.4]{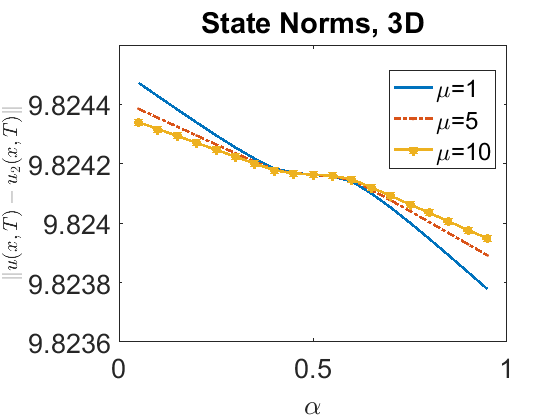}
\caption{Test 6 - Plot of the norms $\|u(T)-u_1(T)\|$ (left) and $\|u(T)-u_2(T)\|$ (right) for the Algorithm 5, with parameters ranging $\mu \in \left\{ 1,\, 5,\, 10 \right\}$ and $\alpha \in \left] 0,1 \right[.$}
\end{figure}

\begin{figure}[H]
\centering
\includegraphics[scale = 0.5]{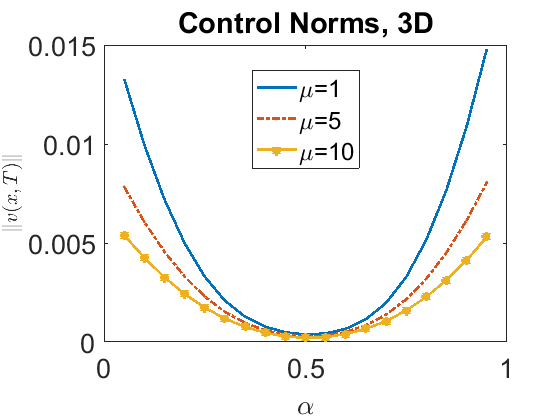}
\caption{Test 6 - Plot of the norms $\|v\|_{L^2\left( \left(0,T\right) \times \omega\right)}$ for parameter values ranging $\mu  \in \left\{ 1,\, 5,\,  10\right\}$ and $\alpha \in \left] 0, 1\right[.$}
\end{figure}

\begin{figure}[H]
\centering
\includegraphics[scale = 0.5]{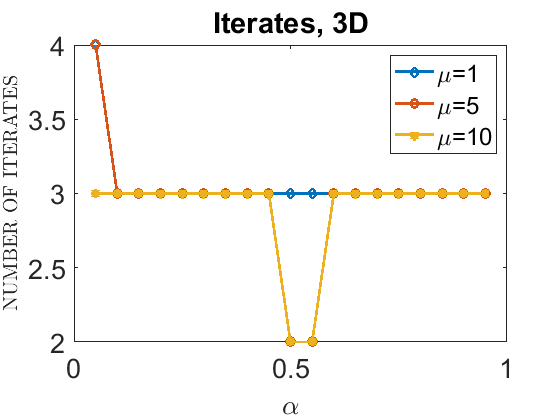}
\caption{Test 6 - Plot of the number of iterates we needed for convergence. Tolerance = $10^{-8}$.}
\end{figure}

\begin{center}
    \textbf{Algorithm 6}
\end{center}

\begin{figure}[H]
\centering
\includegraphics[scale = 0.4]{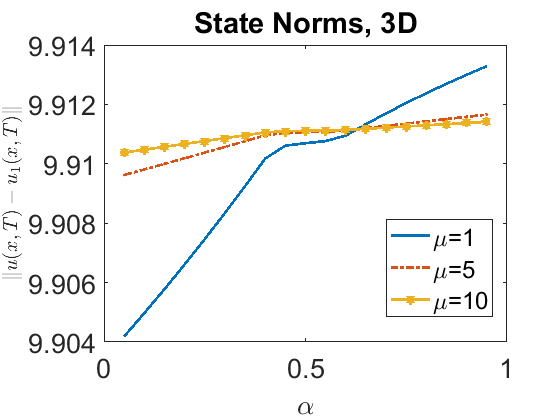}
\includegraphics[scale = 0.4]{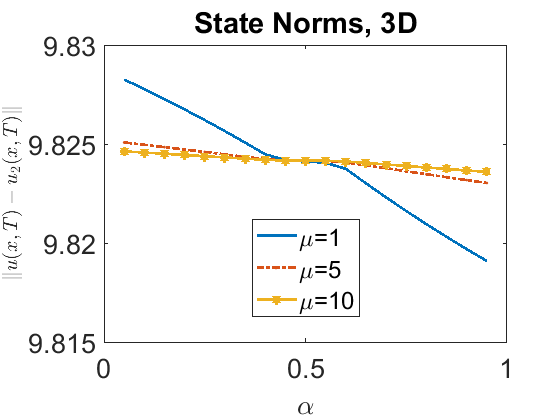}
\caption{Test 6 - Plot of the norms $\|u(T)-u_1(T)\|$ (left) and $\|u(T)-u_2(T)\|$ (right) for the Algorithm 6, with parameters ranging $\mu \in \left\{ 1,\, 5,\, 10 \right\}$ and $\alpha \in \left] 0,1 \right[.$}
\end{figure}

\begin{figure}[H]
\centering
\includegraphics[scale = 0.5]{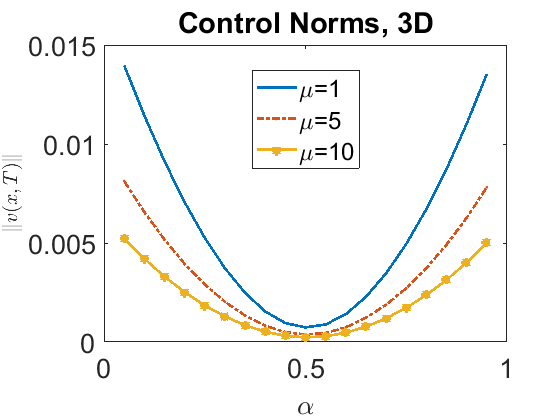}
\caption{Test 6 - Plot of the norms $\|v\|_{L^2\left( \left(0,T\right) \times \omega\right)}$ for parameter values ranging $\mu  \in \left\{ 1,\, 5,\,  10\right\}$ and $\alpha \in \left] 0, 1\right[.$}
\end{figure}

\begin{figure}[H]
\centering
\includegraphics[scale = 0.5]{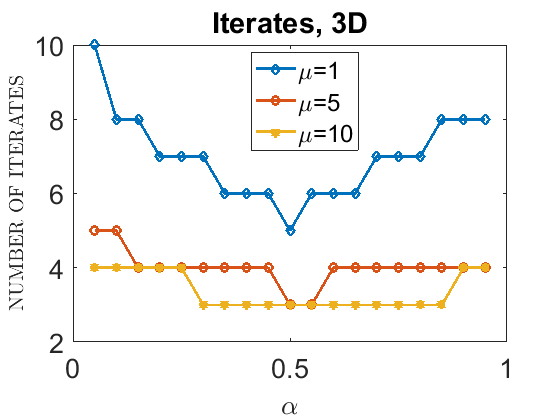}
\caption{Test 6 - Plot of the number of iterates we needed for convergence. Tolerance = $10^{-8}$.}
\end{figure}





\section{Conclusions} \label{sec:Conclusions}

We here analyzed bi-objective optimal control problems for diffusive systems. The concept of optimality we focused on has been that of Pareto. Our models were governed by linear, semilinear and bilinear heat equations. Firstly, we proved in each scenario that the problems were well-posed, under suitable assumptions. Then, we have been able to formulate numerical algorithms allowing to compute the solutions. They are of different kinds: a conjugate gradient method, a fixed-point algorithm and a Newton method.


We finish the work by providing several numerical experiments in two and three dimensions to illustrate our results. 

\section*{Acknowledgement}

This study was financed in part by Coordena\c{c}\~ao de Aperfei\c{c}oamento de Pessoal de N\'ivel Superior - Brasil (CAPES) - Finance code 001. EFC is partially supported by grant PID2020-114976GB-I00 (DGI-MINECO, Spain) and CAPES (Brazil).

\bibliographystyle{plain}

\bibliography{refs}

@book{pareto1964cours,
  title={Cours d'{\'e}conomie politique},
  author={Pareto, Vilfredo},
  volume={1},
  year={1964},
  publisher={Librairie Droz}
}

@article{carvalho2019computation,
  title={On the Computation of Nash and Pareto Equilibria for Some Bi-objective Control Problems},
  author={Carvalho, Pit{\'a}goras P and Fern{\'a}ndez-Cara, Enrique},
  journal={Journal of Scientific Computing},
  volume={78},
  number={1},
  pages={246--273},
  year={2019},
  publisher={Springer}
}

@article{ramos2002nash,
  title={Nash equilibria for the multiobjective control of linear partial differential equations},
  author={Ramos, AM and Glowinski, R and Periaux, J},
  journal={Journal of optimization theory and applications},
  volume={112},
  number={3},
  pages={457--498},
  year={2002},
  publisher={Springer}
}

@incollection{diaz2004approximate,
  title={On the approximate controllability of Stackelberg-Nash strategies},
  author={D{\'\i}az, JI and Lions, JL},
  booktitle={Ocean circulation and pollution control - a mathematical and numerical investigation},
  pages={17--27},
  year={2004},
  publisher={Springer}
}

@article{araruna2015stackelberg,
  title={Stackelberg--Nash exact controllability for linear and semilinear parabolic equations},
  author={Araruna, FD and Fern{\'a}ndez-Cara, E and Santos, MC},
  journal={ESAIM: Control, Optimisation and Calculus of Variations},
  volume={21},
  number={3},
  pages={835--856},
  year={2015},
  publisher={EDP Sciences}
}

@article{limaco2009remarks,
  title={Remarks on hierarchic control},
  author={Limaco, J and Clark, HR and Medeiros, LA},
  journal={Journal of mathematical analysis and applications},
  volume={359},
  number={1},
  pages={368--383},
  year={2009},
  publisher={Elsevier}
}

@article{araruna2017new,
  title={New results on the Stackelberg--Nash exact control of linear parabolic equations},
  author={Araruna, F{\'a}gner D and Fern{\'a}ndez-Cara, Enrique and Guerrero, Sergio and Santos, Mauricio C},
  journal={Systems \& Control Letters},
  volume={104},
  pages={78--85},
  year={2017},
  publisher={Elsevier}
}

@article{guillen2013approximate,
  title={On the approximate controllability of Stackelberg-Nash strategies for Stokes equations},
  author={Guill{\'e}n-Gonz{\'a}lez, F and Marques-Lopes, F and Rojas-Medar, M},
  journal={Proceedings of the American Mathematical Society},
  volume={141},
  number={5},
  pages={1759--1773},
  year={2013}
}

@article{araruna2018stackelberg,
  title={Stackelberg--Nash null controllability for some linear and semilinear degenerate parabolic equations},
  author={Araruna, F{\'a}gner D and Ara{\'u}jo, BSV and Fern{\'a}ndez-Cara, Enrique},
  journal={Mathematics of Control, Signals, and Systems},
  volume={30},
  number={3},
  pages={14},
  year={2018},
  publisher={Springer}
}

@article{carreno2019stackelberg,
  title={Stackelberg--Nash exact controllability for the Kuramoto--Sivashinsky equation},
  author={Carreno, N and Santos, MC},
  journal={Journal of Differential Equations},
  volume={266},
  number={9},
  pages={6068--6108},
  year={2019},
  publisher={Elsevier}
}

@article{diaz2002neumann,
  title={On the von Neumann problem and the approximate controllability of Stackelberg-Nash strategies for some environmental problems.},
  author={D{\'\i}az, Jes{\'u}s Ildefonso},
  journal={RACSAM},
  volume={96},
  number={3},
  pages={343--356},
  year={2002}
}

@incollection{lions1987pareto,
  title={Pareto control of distributed systems. An introduction},
  author={Lions, JL},
  booktitle={Control Problems for Systems Described by Partial Differential Equations and Applications},
  pages={90--104},
  year={1987},
  publisher={Springer}
}

@article{nakoulima2003pareto,
  title={On the pareto control and no-regret control for distributed systems with incomplete data},
  author={Nakoulima, Ousseynou and Omrane, Abdennebi and Velin, Jean},
  journal={SIAM journal on control and optimization},
  volume={42},
  number={4},
  pages={1167--1184},
  year={2003},
  publisher={SIAM}
}

@article{alvarez2010multi,
  title={Multi-objective Pareto-optimal control: an application to wastewater management},
  author={Alvarez-V{\'a}zquez, Lino J and Garc{\'\i}a-Chan, N and Mart{\'\i}nez, Aurea and V{\'a}zquez-M{\'e}ndez, Miguel E},
  journal={Computational Optimization and Applications},
  volume={46},
  number={1},
  pages={135--157},
  year={2010},
  publisher={Springer}
}

@article{desilles2019pareto,
  title={Pareto Front Characterization for Multiobjective Optimal Control Problems Using Hamilton--Jacobi Approach},
  author={D{\'e}silles, Anna and Zidani, Hasnaa},
  journal={SIAM Journal on Control and Optimization},
  volume={57},
  number={6},
  pages={3884--3910},
  year={2019},
  publisher={SIAM}
}

@article{bahaa2003quadratic,
  title={Quadratic Pareto optimal control of parabolic equation with state-control constraints and an infinite number of variables},
  author={Bahaa, GM},
  journal={IMA Journal of Mathematical Control and Information},
  volume={20},
  number={2},
  pages={167--178},
  year={2003},
  publisher={Oxford University Press}
}

@book{allaire2007numerical,
  title={Numerical analysis and optimization: an introduction to mathematical modelling and numerical simulation},
  author={Allaire, Gr{\'e}goire},
  year={2007},
  publisher={Oxford University Press}
}

@book{ciarlet1989introduction,
  title={Introduction to numerical linear algebra and optimisation},
  author={Ciarlet, Philippe G and Miara, Bernadette and Thomas, Jean-Marie},
  year={1989},
  publisher={Cambridge University Press}
}

@article{de2020computation,
  title={On the computation of Nash and Pareto equilibria for some bi-objective control problems for the wave equation},
  author={de Carvalho, Pit{\'a}goras Pinheiro and Fern{\'a}ndez-Cara, Enrique and Ferrel, Juan Bautista L{\'\i}maco},
  journal={Advances in Computational Mathematics},
  volume={46},
  number={5},
  pages={1--30},
  year={2020},
  publisher={Springer}
}

@book{evans10,
  address = {Providence, R.I.},
  author = {Evans, Lawrence C.},
  isbn = {9780821849743 0821849743},
  publisher = {American Mathematical Society},
  refid = {465190110},
  title = {Partial differential equations},
  year = 2010
}

@article{de2020numerical,
  title={Numerical Stackelberg--Nash Control for the Heat Equation},
  author={de Carvalho, Pit\'agoras P and Fern\'andez-Cara, Enrique},
  journal={SIAM Journal on Scientific Computing},
  volume={42},
  number={5},
  pages={A2678--A2700},
  year={2020},
  publisher={SIAM}
}

@article{ramos2019nash,
  title={Nash equilibria strategies and equivalent single-objective optimization problems. The case of linear partial differential equations},
  author={Ramos, Angel Manuel},
  journal={arXiv preprint arXiv:1908.11858},
  year={2019}
}

@article{ramos2002pointwise,
  title={Pointwise control of the Burgers equation and related Nash equilibrium problems: computational approach},
  author={Ramos, A M and Glowinski, R and Periaux, J},
  journal={Journal of optimization theory and applications},
  volume={112},
  number={3},
  pages={499--516},
  year={2002},
  publisher={Springer}
}

@article{fernandez2020theoretical,
  title={Theoretical and numerical results for some bi-objective optimal control problems},
  author={Fern{\'a}ndez-Cara, Enrique and Mar{\'\i}n-Gayte, Irene},
  journal={Communications on Pure \& Applied Analysis},
  volume={19},
  number={4},
  pages={2101},
  year={2020},
  publisher={American Institute of Mathematical Sciences}
}

@article{de2020some,
  title={Some numerical results for control of 3D heat equations using Nash equilibrium},
  author={de Carvalho, Pit{\'a}goras Pinheiro},
  journal={Computational and Applied Mathematics},
  volume={40},
  number={3},
  pages={1--30},
  year={2021},
  publisher={Springer}
}

@article{ramos2007nash,
  title={Nash equilibria in noncooperative predator-prey games},
  author={Ramos, Angel Manuel and Roubicek, Tomas},
  journal={Applied Mathematics and Optimization},
  volume={56},
  number={2},
  pages={211--241},
  year={2007},
  publisher={Springer}
}

@article{dreves2016nash,
  title={A Nash equilibrium approach for multiobjective optimal control problems with elliptic partial differential equations},
  author={Dreves, Axel},
  journal={Control and Cybernetics},
  volume={45},
  number={4},
  year={2016}
}

@article{borzi2013formulation,
  title={Formulation and numerical solution of Nash equilibrium multiobjective elliptic control problems},
  author={Borzi, Alfio and Kanzow, Christian},
  journal={SIAM Journal on Control and Optimization},
  volume={51},
  number={1},
  pages={718--744},
  year={2013},
  publisher={SIAM}
}

@article{dreves2016jointly,
  title={Jointly convex generalized Nash equilibria and elliptic multiobjective optimal control},
  author={Dreves, Axel and Gwinner, Joachim},
  journal={Journal of Optimization Theory and Applications},
  volume={168},
  number={3},
  pages={1065--1086},
  year={2016},
  publisher={Springer}
}

@article{rahman2015fem,
  title={A FEM-Multigrid Scheme for Elliptic Nash-Equilibrium Multiobjective Optimal Control Problems},
  author={Rahman, Mohammad Tanvir and Borz{\`\i}, Alfio},
  journal={Numerical Mathematics: Theory, Methods and Applications},
  volume={8},
  number={2},
  pages={253--282},
  year={2015},
  publisher={Cambridge University Press}
}

@article{peitz2018survey,
  title={A survey of recent trends in multiobjective optimal control-surrogate models, feedback control and objective reduction},
  author={Peitz, Sebastian and Dellnitz, Michael},
  journal={Mathematical and Computational Applications},
  volume={23},
  number={2},
  pages={30},
  year={2018},
  publisher={Multidisciplinary Digital Publishing Institute}
}

@article{banirazi2020heat,
  title={Heat-diffusion: Pareto optimal dynamic routing for time-varying wireless networks},
  author={Banirazi, Reza and Jonckheere, Edmond and Krishnamachari, Bhaskar},
  journal={IEEE/ACM Transactions on Networking},
  year={2020},
  publisher={IEEE}
}

@article{damavandi2017modeling,
  title={Modeling and Pareto based multi-objective optimization of wavy fin-and-elliptical tube heat exchangers using CFD and NSGA-II algorithm},
  author={Damavandi, Mohammad Darvish and Forouzanmehr, Mostafa and Safikhani, Hamed},
  journal={Applied Thermal Engineering},
  volume={111},
  pages={325--339},
  year={2017},
  publisher={Elsevier}
}

@article{chen2014fast,
  title={Fast approach of Pareto-optimal solution recommendation to multi-objective optimal design of serpentine-channel heat sink},
  author={Chen, Yi and Peng, Bei and Hao, Xiaohong and Xie, Gongnan},
  journal={Applied thermal engineering},
  volume={70},
  number={1},
  pages={263--273},
  year={2014},
  publisher={Elsevier}
}

@article{logist2010fast,
  title={Fast Pareto set generation for nonlinear optimal control problems with multiple objectives},
  author={Logist, Filip and Houska, Boris and Diehl, Moritz and Van Impe, Jan},
  journal={Structural and Multidisciplinary Optimization},
  volume={42},
  number={4},
  pages={591--603},
  year={2010},
  publisher={Springer}
}

@article{baradaran2009optimal,
  title={Optimal pareto parametric analysis of two dimensional steady-state heat conduction problems by mlpg method},
  author={Baradaran, GH and MAHMOUDABADI, MJ},
  year={2009},
  publisher={INTERNATIONAL JOURNAL OF ENGINEERING}
}

@article{daniel1971approximate,
  title={The approximate minimisation of functionals Prentice-Hall},
  author={Daniel, J},
  journal={Inc., Englewood Cliffs, NJ},
  year={1971}
}

\end{document}